\documentclass[10pt,twosides]{amsart}
\usepackage{amssymb,amsmath,amsthm, amscd, enumerate, mathrsfs}
\usepackage{graphicx, hhline}
\usepackage[all]{xy}
\usepackage[usenames]{color}
\usepackage{hyperref}
\usepackage{fancyhdr}

\hypersetup{colorlinks=true}

\title{A class of singularity of arbitrary pairs and log canonicalizations}
\author{Kenta Hashizume}
\date{2019/05/20, version 0.14}
\address{Department of Mathematics, Graduate School of Science, 
Kyoto University, Kyoto 606-8502, Japan}
\curraddr{Graduate School of Mathematical Sciences, 
The University of Tokyo, 3-8-1 Komaba Meguro-ku Tokyo 153-8914, Japan}
\email{hkenta@ms.u-tokyo.ac.jp}

\newtheorem{thm}{Theorem}[section]

\newtheorem{lem}[thm]{Lemma}
\newtheorem{cor}[thm]{Corollary}
\newtheorem{prop}[thm]{Proposition}

\theoremstyle{definition}
\newtheorem{defn}[thm]{Definition}
\newtheorem{rem}[thm]{Remark}
\newtheorem{note}[thm]{Notation}
\newtheorem{exam}[thm]{Example}
\newtheorem*{ack}{Acknowledgments} 
\newtheorem*{divisor}{Divisors} 
\newtheorem*{sing}{Singularities of pairs} 
\newtheorem*{model}{Models}

\newtheorem{step}{Step}
\newtheorem{step2}{Step}
\newtheorem{step3}{Step}

\newtheorem*{claim*}{Claim}
\begin{document}

\maketitle

\begin{abstract}
We define a class of singularity on arbitrary pairs of a normal variety and an effective $\mathbb{R}$-divisor on it, which we call pseudo-lc in this paper. 
This is a generalization of the usual lc singularity of pairs and log canonical singularity of normal varieties introduced by de Fernex and Hacon. 
By giving examples of pseudo-lc pairs which are not lc or log canonical in the sense of de Fernex--Hacon's paper, we show that pseudo-lc singularity is a strictly extended notion of those singularities. 
We prove that pseudo-lc pairs admit a small lc modification. 
We also discuss a criterion of log canonicity. 
\end{abstract}

\tableofcontents

\section{Introduction}\label{sec1}

Throughout this paper we will work over the complex number field. 

In the birational geometry, we often deal with not only algebraic varieties but also pairs of an algebraic variety and a divisor. 
Pairs of a variety and a divisor naturally appear, for example, a curve and marked points, or an open variety and the boundary of its compactification. 
Even when we study geometric properties of higher-dimensional algebraic varieties, pairs can be a very powerful tool to work induction on dimension of varieties. 
When we deal with pairs $(X,\Delta)$, we usually assume that the log canonical divisor $K_{X}+\Delta$ is $\mathbb{R}$-Cartier. 
Using this property, we often compare log canonical divisors of two pairs which are birationally equivalent in a sense. 
For example, when we are given pairs $(X,\Delta)$ and $(X',\Delta')$ with a birational map $X\dashrightarrow X'$, we take a common resolution $f\colon Y\to X$ and $f'\colon Y\to X'$ of $X\dashrightarrow X'$ and compare $f^{*}(K_{X}+\Delta)$ and $f'^{*}(K_{X'}+\Delta')$. 
Some classes of pairs with $\mathbb{R}$-Cartier log canonical divisors and mild singularities, such as lc pairs, klt pairs, and so on (see \cite{kollar-mori}), are in particular important to study higher-dimensional algebraic varieties. 
In fact, a lot of important results in the birational geometry were proved in the framework of lc or klt pairs (for example, \cite{bchm}, \cite{fujinonon-van}, \cite{fujino-fund}, \cite{birkar-flip}, \cite{haconxu-lcc}, \cite{hmx}, \cite{birkar-bab}). 

It is difficult to carry out similar arguments on pairs whose log canonical divisors are not $\mathbb{R}$-Cartier. 
In \cite{dfh}, de Fernex and Hacon defined the pullback of arbitrary $\mathbb{Q}$-divisors. 
Using it, they defined relative log canonical divisors, multiplier ideal sheaves and classes of singularities on pairs $(X, \sum a_{i}Z_{i})$ of a normal quasi-projective variety $X$ and a formal $\mathbb{R}_{\geq 0}$-linear combination 
$\sum a_{i}Z_{i}$ of subschemes $Z_{i}\subset X$. 
They proved that multiplier ideal sheaves, log canonical pairs and log terminal pairs in the sense of \cite{dfh} have various properties similar to those on the usual pairs. 
For instance, they proved vanishing theorem of multiplier ideal sheaves and that log terminal singularities have only rational singularities.

In this paper, we study an extension of lc singularity. 
The purpose of this paper is to generalize lc singularity to a class of singularity of  pairs whose log canonical divisor is not necessarily $\mathbb{R}$-Cartier and to investigate relations between the new singularity and lc singularity or log canonical singularity introduced by \cite{dfh}. 

We deal with arbitrary pairs of a normal variety $X$ and an effective $\mathbb{R}$-divisor $\Delta$ on it, which we denote $\langle X,\Delta \rangle$ to distinguish them from pairs whose log canonical divisor is $\mathbb{R}$-Cartier. 
For any prime divisor $P$ over $X$, we define {\em discrepancy} of $P$ with respect to $\langle X,\Delta \rangle$, denoted by $\alpha(P,X,\Delta)$ in this paper (Definition \ref{defnalmostdiscrepancy}), and define {\em pseudo-lc} singularity by using it (Definition \ref{defnalmostpair}). 
We show that $\alpha(\,\cdot\,,X,\Delta)$ is a generalization of the usual discrepancy (Lemma \ref{lembasic}), and that the $b$-divisor defined with $\alpha(\,\cdot\,,X,\Delta)$ is a logarithmic analog of the relative canonical $b$-divisor as in \cite{bdff} (Theorem \ref{thmnefenvelope}). 
In particular, the class of pseudo-lc pairs contains the usual lc pairs and potentially lc pairs (see \cite[Definition 17]{kollar-logpluri}) as special cases. 
Also, we prove that pseudo-lc pairs are closely related to log canonical singularity in the sense of \cite{dfh} (Proposition \ref{propdfh}) and they appear in generalized lc pairs introduced in \cite{bz} (Proposition \ref{propgeneralizedlc}). 
By giving an example of pseudo-lc pairs which are not lc or log canonical in the sense of \cite{dfh} (Example \ref{examnotdfhlc}), we show that pseudo-lc singularity is a strictly extended notion of those singularities. 
Furthermore, for any pair with a boundary $\mathbb{R}$-divisor, we construct a log canonicalization which only extracts bad divisors measured by discrepancy. 
The following theorem is the main result of this paper.

\begin{thm}[=Theorem \ref{thmbirat}]\label{thmmain}
Let $\langle X,\Delta \rangle$ be a pair such that $\Delta$ is a boundary $\mathbb{R}$-divisor. 
Then, there is a projective birational morphism $h\colon W\to X$ from a normal variety $W$ such that
\begin{itemize}
\item
any $h$-exceptional prime divisor $E_{h}$ satisfies $\alpha(E_{h},X,\Delta)<-1$, 
\item
the reduced $h$-exceptional divisor $E_{\rm red}$ is $\mathbb{Q}$-Cartier, and
\item
if we put $\Delta_{W}=h_{*}^{-1}\Delta+E_{\rm red}$, then $K_{W}+\Delta_{W}$ is $\mathbb{R}$-Cartier and the pair $(W,\Delta_{W})$ is lc. 
\end{itemize}
\end{thm}

In the case of pseudo-lc pairs, we have the following theorem: 

\begin{thm}[see Theorem \ref{thmsmalllc}]\label{thmrellc}
Let $\langle X,\Delta \rangle$ be a pseudo-lc pair. 
Then, there is an lc modification $h\colon (W,\Delta_{W})\to X$ such that $h$ is small. 
\end{thm}

For definition of lc modification, see Definition \ref{defnlcmodification} (see also  \cite[Definition 18]{kollar-logpluri}). 
In fact, we prove a stronger result than Theorem \ref{thmrellc}, and with the result we discuss a sufficient condition of log canonicity for arbitrary pairs.  
The key ingredient of the proof of Theorem \ref{thmmain} and Theorem \ref{thmrellc} is the following theorem, a special kind of the relative log MMP. 

\begin{thm}[= Theorem \ref{thmrelmmp}]\label{thm1.3}
Let $\pi\colon X\to Z$ be a projective morphism of normal quasi-projective varieties, and let $(X,\Delta)$ be an lc pair. 
Suppose that
\begin{itemize}
\item
$-(K_{X}+\Delta)$ is pseudo-effective over $Z$, and 
\item 
for any lc center $S$ of $(X,\Delta)$ and its normalization $S^{\nu}\to S$, 
the pullback of $-(K_{X}+\Delta)$ to $S^{\nu}$ is pseudo-effective over $Z$.  
\end{itemize}
Then, $(X,\Delta)$ has a good minimal model or a Mori fiber space over $Z$. 
\end{thm}

The proof of Theorem \ref{thm1.3} is in Section \ref{sec3}, and the proof of the main result is in Section \ref{sec4}.  
To prove the main result for a given pair $\langle X,\Delta \rangle$, we take a log resolution $Y\to X$ of $\langle X,\Delta \rangle$, then we run a relative log MMP  for a lc pair $(Y,\Delta_{Y})$ and apply Theorem \ref{thm1.3} to construct a log canonical model of $(Y,\Delta_{Y})$ over $X$. 
In our situation, known results (for example, results in \cite{birkar-flip}, \cite{haconxu-lcc} and \cite{has-mmp}) 
are insufficient for the termination of the log MMP because lc centers of $(Y,\Delta_{Y})$ have only weak property. 
Therefore, we need to establish a new relative log MMP in more general setting. 
Theorem \ref{thm1.3} is suitable for our situation, and it plays a crucial role in the proof of the main result. 

We note that the notions of pseudo-lc pairs and lc pairs coincide in the case of surfaces (Corollary \ref{corsurface}), and pseudo-lc pairs in Example \ref{examnotlc} or Example \ref{examnotdfhlc} include threefolds. 
So a gap between pseudo-lc singularity and lc singularity or log canonical singularity in the sense of \cite{dfh} arises when the dimension of the variety is greater than $2$.  

By Theorem \ref{thmrellc}, we obtain two important theorem on pseudo-lc pairs.

\begin{thm}[=Theorem \ref{thmfinite}]
Let $\langle X,\Delta \rangle$ be a pseudo-lc pair such that $\Delta$ is a $\mathbb{Q}$-divisor. 
Then, the graded sheaf of $\mathcal{O}_{X}$-algebra $\bigoplus_{m\geq0}\mathcal{O}_{X}(\llcorner m(K_{X}+\Delta)\lrcorner)$ is finitely generated. 
If $X$ is projective and the minimal model theory holds, then the log canonical ring  
$\bigoplus_{m\geq0}H^{0}(X, \mathcal{O}_{X}(\llcorner m(K_{X}+\Delta)\lrcorner))$
is a finitely generated $\mathbb{C}$-algebra. 
\end{thm}

\begin{thm}[=Theorem \ref{thmkodaira}, Kodaira type vanishing theorem]\label{thmvanish}
Let $\pi\colon X\to Z$ be a projective morphism of normal varieties and $\langle X,\Delta \rangle$ be a pseudo-lc pair. 
Let $D$ be a Weil divisor on $X$ such that $D-(K_{X}+\Delta)$ is $\pi$-ample. 

Then, $R^{i}\pi_{*}\mathcal{O}_{X}(D)=0$ for any $i>0$. 
\end{thm}

In Section \ref{sec5}, we study gaps between pseudo-lc and lc singularities in detail. 
As an application of Theorem \ref{thmsmalllc}, which is a strong version of Theorem \ref{thmrellc}, we prove the following theorem: 

\begin{thm}[=Theorem \ref{thm5.1}]\label{thmlccondition}
Let $X$ be a normal quasi-projective variety, and let $\Delta$ be a boundary $\mathbb{R}$-divisor. 
\begin{enumerate}
\item[(1)]
There is $\mathfrak{D}_{1}$ a finite set of prime divisors over $X$ such that if 
$${\rm sup}\{a(P,X,\Delta+G)|\,G\geq0, K_{X}+\Delta+G{\rm \; is\;}\mathbb{R}{\rm \mathchar`-Cartier}\}\geq-1$$ 
for all $P\in \mathfrak{D}_{1}$, then $\langle X,\Delta \rangle$ has a small lc modification. 
In particular, when $\Delta$ is a $\mathbb{Q}$-divisor, the graded sheaf of $\mathcal{O}_{X}$-algebra $\bigoplus_{m\geq0}\mathcal{O}_{X}(\llcorner m(K_{X}+\Delta)\lrcorner)$ is finitely generated. 
\item[(2)]
Suppose that $\langle X,\Delta \rangle$ has a small lc modification. 
Let $x \in X$ be a closed point. 
Then, there is $\mathfrak{D}_{2}$ a finite set of prime divisors over $X$ such that $K_{X}+\Delta$ is $\mathbb{R}$-Cartier and $( X,\Delta )$ is lc in a neighborhood of $x$ if and only if the following relation holds for any $P\in \mathfrak{D}_{2}$. 
\begin{equation*} \begin{split} &{\rm sup}\{a(P,X,\Delta+G)|\,G\geq0, K_{X}+\Delta+G{\rm \; is\;}\mathbb{R}{\rm \mathchar`-Cartier}\}\\ =&{\rm inf}\{a(P,X,\Delta-G')|\,G'\geq0, K_{X}+\Delta-G'{\rm \; is\;}\mathbb{R}{\rm \mathchar`-Cartier}\}. \end{split} \end{equation*}
\end{enumerate}
\end{thm}

Theorem \ref{thmlccondition} gives a way to check local log canonicity in two steps by using the usual discrepancies of finitely many prime divisors.  

As a corollary of Theorem \ref{thmlccondition}, we obtain a necessary and sufficient condition of log canonicity for pseudo-lc pairs. 

\begin{thm}[=Corollary \ref{corlccriterion}]\label{thm1.7}
Let $\langle X,\Delta \rangle$ be a pair such that $X$ is quasi-projective. 
Then, $K_{X}+\Delta$ is $\mathbb{R}$-Cartier and $( X,\Delta )$ is lc if and only if $\langle X,\Delta \rangle$ is pseudo-lc and the following equation holds for any prime divisor $P$ over $X$. 
\begin{equation*}
\begin{split}
&{\rm sup}\{a(P,X,\Delta+G)|\,G\geq0, K_{X}+\Delta+G{\rm \; is\;}\mathbb{R}{\rm \mathchar`-Cartier}\}\\
=&{\rm inf}\{a(P,X,\Delta-G')|\,G'\geq0, K_{X}+\Delta-G'{\rm \; is\;}\mathbb{R}{\rm \mathchar`-Cartier}\}.
\end{split}
\end{equation*} 
\end{thm}

We also give the proof of Theorem \ref{thm1.7} using the notion of numerically Cartier divisors (see \cite[Definition 5.2]{bdffu}). 
As we will see, Theorem \ref{thm1.7} can be regarded as an lc analog of \cite[Corollary 5.17]{bdffu}. 
We would like to remark that the proof is also an application of Theorem \ref{thm1.3} (or Lemma \ref{lemneflogabund}). 
For details, see Section \ref{sec5}. 

The contents of this paper are as follows: 
In Section \ref{sec2}, we collect definitions and some results on the log MMP. 
In Section \ref{sec3}, we show a special kind of the relative log MMP, which is a generalization of \cite[Theorem 1.1]{has-mmp}. 
In Section \ref{sec4}, which is the main part of this paper, we define pseudo-lc singularity and prove basic properties of pseudo-lc pairs, the main theorem and other results. 
In Section \ref{sec5}, we prove Theorem \ref{thmlccondition} and Theorem \ref{thm1.7}.

\begin{ack}
The author was partially supported by JSPS KAKENHI Grant Number JP16J05875. 
The topic of this paper came from a discussion with Professor Yuji Odaka. 
The author would like to thank him for answering questions, informing the author of the paper \cite{dfh}, and giving comments. 
The author is grateful to Professor Kento Fujita for fruitful discussions, answering questions and giving comments. 
The author is grateful to Professor Osamu Fujino for comments on previous version of  Theorem \ref{thmlccondition} and Theorem \ref{thm1.7}. 
The author thanks Professor J\'anos Koll\'ar for comments. 
\end{ack}

\section{Preliminaries}\label{sec2}
In this section, we collect definitions and some important theorems. 

\subsection{Definitions}
We collect some definitions. 

\begin{divisor}
Let $\pi\colon X\to Z$ be a projective morphism of normal varieties. 
We use the standard definition of $\pi$-nef $\mathbb{R}$-divisor, $\pi$-ample $\mathbb{R}$-divisor, $\pi$-semi-ample $\mathbb{R}$-divisor, $\pi$-big $\mathbb{R}$-divisor and $\pi$-pseudo-effective $\mathbb{R}$-divisor.
\end{divisor}

\begin{sing}
In this paper, we deal with two kinds of pairs. 

We recall definition of the usual pairs. 
A {\em sub-pair} $(X,\Delta)$ consists of a normal variety $X$ and an $\mathbb{R}$-divisor $\Delta$ such that $K_{X}+\Delta$ is $\mathbb{R}$-Cartier. 
When $\Delta$ is effective, we call $(X,\Delta)$ a {\em pair}. 
When coefficients of $\Delta$ belong to $[0,1]$, the divisor $\Delta$ is called a {\em boundary divisor}. 
When we write $(X,{\rm Supp}\Delta)$, we pay attention to $X$ and the support of $\Delta$. 
Therefore, $(X,{\rm Supp}\Delta)$ simply denotes a pair of a variety and a subscheme of pure codimension one. 

Let $(X,\Delta)$ be a sub-pair and let $P$ be a prime divisor over $X$, that is, a prime divisor on a normal variety $Y$ with a projective birational morphism $Y\to X$. 
Then, $a(P,X,\Delta)$ denotes the discrepancy of $P$ with respect to $(X,\Delta)$. 
When $(X,\Delta)$ is a pair, we use definitions of Kawamata log terminal (klt, for short) pair, log canonical (lc, for short) pair and divisorially log terminal (dlt, for short) pair as in \cite{kollar-mori}. 
An {\em lc center} of $(X,\Delta)$ is the image on $X$ of a prime divisor $P$ over $X$ satisfying $a(P,X,\Delta)=-1$. 

We also deal with arbitrary pairs of a normal variety $X$ and an effective $\mathbb{R}$-divisor $\Delta$ on it. 
When we do not assume that $K_{X}+\Delta$ is $\mathbb{R}$-Cartier, we denote the pair of $X$ and $\Delta$ by $\langle X,\Delta \rangle$ to distinguish from the usual pairs. 
\end{sing}

\begin{model}\label{deflogbir}
We use the definition of weak lc model, log minimal model, good minimal model and Mori fiber space as in \cite[Section 2]{birkar-flip}. 
We freely use the result of dlt blow-up for usual pairs (see, for example, \cite[Theorem 4.4.21]{fujino-book})
\end{model}

\begin{rem}\label{remmodels}
Let $(X,\Delta)$ be an lc pair and $(X',\Delta')$ be a log minimal model of $(X,\Delta)$. 
Let $(X'',\Delta'')$ be a $\mathbb{Q}$-factorial dlt pair such that $K_{X''}+\Delta''$ is nef, $X''$ and $X'$ are isomorphic in codimension one, and $\Delta''$ is the birational transform of $\Delta'$ on $X''$. 
Then, $(X'',\Delta'')$ is also a log minimal model of $(X,\Delta)$. 
Moreover, if $(X',\Delta')$ is a good minimal model of $(X,\Delta)$, then $(X'',\Delta'')$ is  a good minimal model of $(X,\Delta)$. 
\end{rem}

\begin{defn}[Log canonical model]\label{defnlcmodel}
Let $X\to Z$ be a projective morphism from a normal variety to a variety, and let $(X,\Delta)$ be an lc pair. 
A weak log canonical model $(X',\Delta')$ of $(X,\Delta)$ over $Z$ is a {\em log canonical model} if $K_{X'}+\Delta'$ is ample over $Z$. 
\end{defn}

\begin{defn}[Lc modification, {\cite[Definition 18]{kollar-logpluri}}]\label{defnlcmodification}
Let $\langle X,\Delta \rangle$ be a pair such that $\Delta$ is a boundary $\mathbb{R}$-divisor. 
Let $f \colon Y\to X$ be a projective birational morphism from a normal variety $Y$, and let $\Gamma$ be the sum of $f_{*}^{-1}\Delta$ and all $f$-exceptional prime divisors with coefficients $1$. 
Then the pair $\langle Y,\Gamma \rangle$ is an {\em lc modification} of $\langle X,\Delta \rangle$ if $K_{Y}+\Gamma$ is an $f$-ample $\mathbb{R}$-Cartier divisor and the pair $(Y,\Gamma)$ is lc. 
An lc modification $(Y,\Gamma)$ of $\langle X,\Delta \rangle$ is {\em small} if $f \colon Y\to X$ is small. 
\end{defn}

From definition, an lc modification is unique up to isomorhpism if it exists. 

\subsection{Results related to the log MMP}
In this subsection, we collect three results on the log MMP. 

In this paper, we use the following two results without any mention. 

\begin{thm}[{\cite[Theorem 4.1]{birkar-flip}}]\label{thmtermi}
Let $(X,\Delta)$ be a $\mathbb{Q}$-factorial lc pair such that $(X,0)$ is klt. 
Let $\pi\colon X \to Z$ be a projective morphism of normal quasi-projective varieties.
 
If there is a log minimal model of $(X,\Delta)$ over $Z$, any $(K_{X}+\Delta)$-MMP over $Z$ with scaling of an ample divisor terminates. 
\end{thm}

\begin{lem}[{\cite[Lemma 2.15]{has-mmp}}]\label{lembirequiv}
Let $\pi\colon X \to Z$ be a projective morphism of normal quasi-projective varieties, and let $(X,\Delta)$ be an lc pair. 
Let $(Y,\Gamma)$ be an lc pair such that there is a projective birational morphism $f\colon Y\to X$ and we can write $K_{Y}+\Gamma=f^{*}(K_{X}+\Delta)+E$ with an $f$-exceptional divisor $E\geq0$. 

Then, $(X,\Delta)$ has a weak lc model (resp.~a log minimal model, a good minimal model) over $Z$ if and only if  $(Y,\Gamma)$ has a weak lc model (resp.~a log minimal model, a good minimal model) over $Z$. 
\end{lem}

We close this section with the following lemma. 
It plays an important role in the proof of Theorem \ref{thmmain}. 

\begin{lem}\label{lemlcmodel} Let $\pi\colon X\to Z$ be a projective morphism of normal varieties, which are not necessarily quasi-projective. 
Let $(X,\Delta)$ be a $\mathbb{Q}$-factorial lc pair such that $(X,0)$ is klt, and let $D$ be an $\mathbb{R}$-divisor on $X$ such that $(X,\Delta+D)$ is lc. 
Suppose that $(X,\Delta+tD)$ has the log canonical model over $Z$ for any $0\leq t<1$.

Then, there is a birational contraction $\phi\colon X\dashrightarrow Y$ over $Z$ and a positive real number $t_{0}$ such that for any $0<t\leq t_{0}$, the pair $(Y,\Delta_{Y}+tD_{Y})$ is the log canonical model of $(X,\Delta+tD)$ over $Z$, where $\Delta_{Y}$ and $D_{Y}$ are the birational transforms of $\Delta$ and $D$ on $Y$, respectively. 
In particular, $D_{Y}$ is $\mathbb{R}$-Cartier. 
\end{lem}

\begin{proof}
Note that the divisor $K_{X}+\Delta$ is big over $Z$ by Definition \ref{defnlcmodel}. 

First, we prove the lemma in the case when $Z$ is quasi-projective. 
Since the log canonical model is in particular a weak lc model with semi-ample log canonical divisor, $(X,\Delta+tD)$ has a good  minimal model over $Z$ for any $0\leq t<1$. 
Let $(X,\Delta)\dashrightarrow (X',\Delta')$ be a sequence of steps of the $(K_{X}+\Delta)$-MMP over $Z$ to a good minimal model, and $X'\to Y_{0}$ be the contraction over $Z$ induced by $K_{X'}+\Delta'$, where $\Delta'$ is the birational transform of $\Delta$ on $X'$. 
Let $D'$ (resp.~$\Delta_{Y_{0}}$) be the birational transform of $D$ (resp.~$\Delta$) on $X'$ (resp.~$Y_{0}$). 
By construction, $K_{Y_{0}}+\Delta_{Y_{0}}$ is ample over $Z$. 
We can find $0<t'_{0}<1$ such that the birational map $X\dashrightarrow X'$ is a sequence of steps of the $(K_{X}+\Delta+t'D)$-MMP for any $0\leq t'\leq t'_{0}$. 
Since $(X,\Delta+tD)$ has a good  minimal model over $Z$ for any $0\leq t<1$, we can run the $(K_{X'}+\Delta'+t'_{0}D')$-MMP over $Z$ and get a good minimal model $(X',\Delta'+t'_{0}D')\dashrightarrow (X'',\Delta''+t'_{0}D'')$. 
By the argument of the length of extremal rays and replacing $t'_{0}$ if necessary, we may assume that  in each step of the $(K_{X'}+\Delta'+t'_{0}D')$-MMP, the birational transform of $K_{X'}+\Delta'$ is trivial over the exremal contraction. 
Since $K_{X'}+\Delta'$ is the pullback of $K_{Y_{0}}+\Delta_{Y_{0}}$, which is ample over $Z$, we see that the $(K_{X'}+\Delta'+t'_{0}D')$-MMP is the $(K_{X'}+\Delta'+t'_{0}D')$-MMP over $Y_{0}$. 
So we have the following diagram. 
$$
\xymatrix
{
X\ar@{-->}[r]\ar[dr]_{\pi}&X'\ar@{-->}[rr]\ar[dr]&&X''\ar[dl]\\
&Z&Y_{0}\ar[l]
}
$$
Since the divisor $K_{X''}+\Delta''+t'_{0}D''$ is semi-ample over $Z$, it is semi-ample over $Y_{0}$. 
Let $X''\to Y$ be the contraction over $Y_{0}$ induced by $K_{X''}+\Delta''+t'_{0}D''$. 
Let $g\colon Y\to Y_{0}$ be the natural morphism, and $\Delta_{Y}$ and $D_{Y}$ be the birational transforms of $\Delta$ and $D$ on $Y$, respectively. 
Then, we have  $K_{Y}+\Delta_{Y}=g^{*}(K_{Y_{0}}+\Delta_{Y_{0}})$, and the divisor $K_{Y}+\Delta_{Y}+t'_{0}D_{Y}$ is ample over $Y_{0}$. 
Since $K_{Y_{0}}+\Delta_{Y_{0}}$ is ample over $Z$, we can find $t_{0}$ such that $0<t_{0}< t'_{0}$ and for any $0<t\leq t_{0}$, the divisor
$$K_{Y}+\Delta_{Y}+tD_{Y}=\frac{t}{t'_{0}}(K_{Y}+\Delta_{Y}+t'_{0}D_{Y})+\left(1-\frac{t}{t'_{0}}\right)g^{*}(K_{Y_{0}}+\Delta_{Y_{0}})$$
is ample over $Z$. 
By construction, for any $0<t\leq t_{0}$, the birational map $X\dashrightarrow X''$ is a sequence of steps of the $(K_{X}+\Delta+tD)$-MMP over $Z$. 
Since $K_{X''}+\Delta''+tD''$ is the pullback of $K_{Y}+\Delta_{Y}+tD_{Y}$, we see that the pair $(Y,\Delta_{Y}+tD_{Y})$ is the log canonical model of $(X,\Delta+tD)$ over $Z$ for any $0<t\leq t_{0}$.
Therefore, the lemma holds true when $Z$ is quasi-projective. 

From now on, we prove the general case. 
We cover $Z$ by a finitely many affine open subset $\{U_{i}\}_{i}$, and we put $V_{i}=\pi^{-1}(U_{i})$. 
By the quasi-projective case of the lemma, for each $i$, there is $t_{i}>0$ and a birational contraction  $V_{i}\dashrightarrow Y_{i}$ over $U_{i}$ such that for any $0<t\leq t_{i}$, the pair $(Y_{i},\Delta_{Y_{i}}+tD_{Y_{i}})$ is the log canonical model of $(V_{i}, \Delta|_{V_{i}}+tD|_{V_{i}})$ over $U_{i}$. 
Set $t''={\rm min}\{t_{i}\}_{i}$ and construct $Y$ by gluing all $Y_{i}$. 
By construction, for any $0<t \leq t''$, the pair $(Y,\Delta_{Y}+tD_{Y})$ is the log canonical model of $(X,\Delta+tD)$ over $Z$. 
Therefore, the birational map $X\dashrightarrow Y$ over $Z$ is the desired one. 
\end{proof}

\section{A spacial kind of relative log MMP}\label{sec3}

In this section, we show a special kind of the relative log MMP (Theorem \ref{thmrelmmp}), which plays a crucial role in the proof of Theorem \ref{thmmain}. 

\begin{defn}
Let $X$ be a normal projective variety, and let $D$ be an $\mathbb{R}$-Cartier $\mathbb{R}$-divisor $D$ on $X$. 

First, we define the {\em invariant  Iitaka dimension} of $D$, denoted by $\kappa_{\iota}(X,D)$, as follows (see also \cite[Definition 2.5.5]{fujino-book}):  
If there is an $\mathbb{R}$-divisor $E\geq 0$ such that $D\sim_{\mathbb{R}}E$, set $\kappa_{\iota}(X,D)=\kappa(X,E)$. 
Here, the right hand side is the usual Iitaka dimension of $E$. 
Otherwise, we set $\kappa_{\iota}(X,D)=-\infty$. 
We can check that $\kappa_{\iota}(X,D)$ is well-defined, i.e., when there is $E\geq 0$ such that $D\sim_{\mathbb{R}}E$, $\kappa_{\iota}(X,D)$ does not depend on the choice of $E$. 
By definition, we have $\kappa_{\iota}(X,D)\geq0$ if and only if $D$ is $\mathbb{R}$-linearly equivalent to an effective $\mathbb{R}$-divisor. 

Next, we define the {\em numerical dimension} of $D$, denoted by $\kappa_{\sigma}(X,D)$, as follows (see also \cite[V, 2.5 Definition]{nakayama}): 
For any Cartier divisor $A$ on $X$, we set
$$
\sigma(D;A)={\rm max}\left\{k\in \mathbb{Z}_{\geq0}\middle|\, \underset{m\to \infty}{\rm lim}{\rm sup}\frac{{\rm dim}H^{0}(X,\mathcal{O}_{X}(\llcorner mD \lrcorner+A))}{m^{k}}>0\right\}
$$
if ${\rm dim}H^{0}(X,\mathcal{O}_{X}(\llcorner mD \lrcorner+A))>0$ for infinitely many $m>0$, and otherwise we set $\sigma(D;A):=-\infty$. 
Then, we define 
$$\kappa_{\sigma}(X,D):={\rm max}\{\sigma(D;A)\,|\,A{\rm\; is\; a\;Cartier\;divisor\;on\;}X\}.$$

Let $X\to Z$ be a projective morphism from a normal variety to a variety, and let $D$ be an $\mathbb{R}$-Cartier $\mathbb{R}$-divisor on $X$. 
Then, the {\em relative numerical dimension} of $D$ over $Z$ is defined by $\kappa_{\sigma}(F,D|_{F})$, where $F$ is a sufficiently general fiber of the Stein factorization of $X\to Z$ (see \cite[2.2]{has-mmp}). 
\end{defn}

\begin{rem}\label{remdiv}
We write down basic properties of the invariant Iitaka dimension and the numerical dimension. 
\begin{enumerate}
\item
Let $D_{1}$ and $D_{2}$ be $\mathbb{R}$-Cartier $\mathbb{R}$-divisors on a normal projective variety $X$. 
\begin{itemize}
\item
Suppose that $D_{1}\sim_{\mathbb{R}}D_{2}$.
Then, we have $\kappa_{\iota}(X,D_{1})=\kappa_{\iota}(X,D_{2})$ and $\kappa_{\sigma}(X,D_{1})=\kappa_{\sigma}(X, D_{2})$. 
\item
Suppose that we have $D_{1}\sim_{\mathbb{R}}N_{1}$ and $D_{2}\sim_{\mathbb{R}}N_{2}$ for  $\mathbb{R}$-divisors $N_{1}\geq0$ and $N_{2}\geq0$ respectively such that ${\rm Supp}N_{1}={\rm Supp}N_{2}$. 
Then, we have $\kappa_{\iota}(X,D_{1})=\kappa_{\iota}(X,D_{2})$ and $\kappa_{\sigma}(X,D_{1})=\kappa_{\sigma}(X, D_{2})$. 
\end{itemize}
\item
Let $f\colon Y \to X$ be a surjective morphism of normal projective varieties and $D$ an $\mathbb{R}$-Cartier $\mathbb{R}$-divisor on $X$. 
\begin{itemize}
\item
We have $\kappa_{\iota}(X,D)=\kappa_{\iota}(Y,f^{*}D)$ and $\kappa_{\sigma}(X,D)=\kappa_{\sigma}(Y,f^{*}D)$.
\item
Suppose that $f$ is birational. 
Let $D'$ be an $\mathbb{R}$-Cartier $\mathbb{R}$-divisor on $Y$ such that  $D'=f^{*}D+E$ for some effective $f$-exceptional divisor $E$. 
Then, we have $\kappa_{\iota}(X,D)=\kappa_{\iota}(Y,D')$ and $\kappa_{\sigma}(X,D)=\kappa_{\sigma}(Y, D')$.
\end{itemize}
\end{enumerate}
\end{rem}

\begin{defn}[Relatively abundant and relatively log abundant divisor]\label{defnabund}
Let $\pi\colon X\to Z$ be a projective morphism from a normal variety to a variety, and let $D$ be an $\mathbb{R}$-Cartier $\mathbb{R}$-divisor on $X$. 
We say $D$ is $\pi$-{\em abundant} or {\em abundant over} $Z$ if the equality $\kappa_{\iota}(F,D|_{F})=\kappa_{\sigma}(F,D|_{F})$ holds,  where $F$ is a sufficiently general fiber of the Stein factorization of $\pi$. 

Let $\pi\colon X\to Z$ and $D$ be as above, and let $(X,\Delta)$ be an lc pair. 
We say $D$ is $\pi$-{\em log abundant} with respect to $(X,\Delta)$ when $D$ is $\pi$-abundant and the pullback of $D$ to the normalization of any lc center of $(X,\Delta)$ is abundant over $Z$. 
\end{defn}

The following lemma is the $\mathbb{R}$-divisor version of \cite[Theorem 4.12]{fujino-gongyo}. 

\begin{lem}
\label{lemneflogabund}
Let $\pi\colon X\to Z$ be a morphism of normal projective varieties, and let $(X,\Delta)$ be an lc pair such that $\Delta$ is an $\mathbb{R}$-divisor. 
Suppose that $K_{X}+\Delta$ is $\pi$-nef and $\pi$-log abundant with respect to $(X,\Delta)$. 

Then, $K_{X}+\Delta$ is $\pi$-semi-ample. 
\end{lem}

\begin{proof}
We prove the lemma when $(X,\Delta)$ is not klt because the klt case of the lemma can be proved with a very similar idea to non-klt case and a simpler argument than the proof of non-klt case. 
By adding the pullback of a sufficiently ample divisor on $Z$, we may assume that the divisor $K_{X}+\Delta$ is globally nef and log abundant with respect to $(X,\Delta)$. 
We show that $K_{X}+\Delta$ is semi-ample by induction on ${\rm dim}\,X$.  
So we may assume that $Z$ is a point. 

By taking a dlt blow-up, we may assume that $(X,\Delta)$ is $\mathbb{Q}$-factorial dlt. 
Since $K_{X}+\Delta$ is abundant, there is $N\geq0$ such that $K_{X}+\Delta\sim_{\mathbb{R}}N$. 
Let $\mathcal{L}\subset{\rm WDiv}_{\mathbb{R}}(X)$ be the set of boundary $\mathbb{R}$-divisors $\Delta'$ such that $(X,\Delta')$ is lc, ${\rm Supp}\Delta'={\rm Supp}\Delta$ and $\llcorner \Delta' \lrcorner=\llcorner \Delta \lrcorner$. 
By an argument of convex geometry, we see that the set
\begin{equation*}\left\{ \Delta'\in \mathcal{L}
 \left|\begin{array}{l}
 \bullet \; (X,\Delta'){\rm \; is \; dlt,\;}\\
 \bullet \; K_{X}+\Delta' {\rm \; is \; nef,\;and\;}\\
 \bullet \; K_{X}+\Delta'\sim_{\mathbb{R}}N'  {\rm \; for \; an \;}N'\geq0{\rm \; such \; that \;} {\rm Supp}N'={\rm Supp}N.
 \end{array}\right.\right\}\end{equation*}
contains a rational polytope $\mathcal{T}_{(X)}\subset \mathcal{L}$ in which $\Delta$ is contained. 
By shrinking $\mathcal{T}_{(X)}$, we can assume that lc centers of $(X,\Delta')$ coincide with those of $(X,\Delta)$ for any $\Delta'\in \mathcal{T}_{(X)}$. 
By Remark \ref{remdiv} (1), $K_{X}+\Delta'$ is abundant for any $\Delta'\in \mathcal{T}_{(X)}$.  

We fix an lc center $S$ of $(X,\Delta)$. 
Note that $S$ is also an lc center of $(X,\Delta')$ for any $\Delta'\in \mathcal{T}_{(X)}$. 
By construction, any divisor $\Delta'\in \mathcal{L}$ can be written as $\llcorner \Delta \lrcorner +\sum_{i} d_{i}'D_{i}$, where $0\leq d_{i}'<1$ and $D_{i}$ are prime divisors which are components of $\Delta-\llcorner \Delta \lrcorner$. 
Then ${\rm Supp}D_{i}\nsupseteq S$. 
Since $X$ is $\mathbb{Q}$-factorial, for any component $D_{i}$ of $\Delta-\llcorner \Delta \lrcorner$, the restriction $D_{i}|_{S}$ is well-defined as an effective $\mathbb{Q}$-Cartier $\mathbb{Q}$-divisor on $S$.  
We also see that the divisor $(K_{X}+\llcorner \Delta \lrcorner)|_{S}$ is a $\mathbb{Q}$-Cartier $\mathbb{Q}$-divisor on $S$. 
We define an $\mathbb{R}$-divisor $\Delta_{S}$ by adjunction $K_{S}+\Delta_{S}=(K_{X}+\Delta )|_{S}$. 
Then, $K_{S}+\Delta_{S}$ is semi-ample by the induction hypothesis. 
Therefore, if we write $\Delta=\llcorner \Delta \lrcorner +\sum d_{i}D_{i}$ with real numbers $0< d_{i}<1$, we have 
$K_{S}+\Delta_{S}=(K_{X}+\llcorner \Delta \lrcorner)|_{S}+\sum_{i} d_{i} (D_{i}|_{S})$
 and it can be written as an $\mathbb{R}_{>0}$-linear combination of finitely many (not necessarily effective) semi-ample $\mathbb{Q}$-divisors $\{A_{j}\}_{j}$. 
We can write $(K_{X}+ \Delta')|_{S}=(K_{X}+\llcorner \Delta \lrcorner)|_{S}+\sum_{i} d_{i}' (D_{i}|_{S})$ for any $\Delta'\in \mathcal{L}$. 
From these facts and an argument of convex geometry, the set
\begin{equation*}\left\{ \Delta'\in \mathcal{T}_{(X)}
 \left|\begin{array}{l}(K_{X}+\Delta')|_{S}=\sum_{j}a_{j}A_{j}{\rm ,\,where\;}a_{j}\in \mathbb{R}_{\geq 0}\!\!
\end{array}\right.\right\}\end{equation*}
contains a rational polytope $\mathcal{T}_{(S)}\ni \Delta$. 

We consider 
$$\mathcal{T}=\bigcap
_{\substack {S:{\rm \;lc \;center}\\{\rm \quad\; of\;}(X,\Delta)}}
\mathcal{T}_{(S)},$$
which is a rational polytope containing $\Delta$. 
We pick positive real numbers $r_{1},\,\cdots,r_{m}$ and $\mathbb{Q}$-divisors $\Delta^{(1)},\,\cdots,\Delta^{(m)}\in \mathcal{T}$ such that $\sum_{k=1}^{m} r_{k}=1$ and $\sum_{k=1}^{m} r_{k}\Delta^{(k)}=\Delta.$ 
By construction of $\mathcal{T}$, for any $\Delta'\in \mathcal{T}$, the divisor $K_{X}+\Delta'$ is nef and log abundant with respect to $(X,\Delta')$. 
By \cite[Theorem 4.12]{fujino-gongyo}, $K_{X}+\Delta^{(k)}$ are semi-ample. 
Since $K_{X}+\Delta=\sum_{k=1}^{m} r_{k}(K_{X}+\Delta^{(k)})$, we see that $K_{X}+\Delta$ is semi-ample. 
So we complete the proof. 
\end{proof}

\begin{thm}\label{thmrelmmp}
Let $\pi\colon X\to Z$ be a projective morphism of normal quasi-projective varieties and $(X,\Delta)$ be an lc pair. 
Suppose that
\begin{itemize}
\item
$-(K_{X}+\Delta)$ is pseudo-effective over $Z$, and 
\item 
for any lc center $S$ of $(X,\Delta)$ and its normalization $S^{\nu}\to S$, 
the pullback of $-(K_{X}+\Delta)$ to $S^{\nu}$ is pseudo-effective over $Z$.  
\end{itemize}
Then, $(X,\Delta)$ has a good minimal model or a Mori fiber space over $Z$. 
\end{thm}

\begin{proof}
We may assume that $(X,\Delta)$ is not klt because otherwise the theorem follows from  \cite[Theorem 1.2]{has-mmp}. 
We prove Theorem \ref{thmrelmmp} by induction on the dimension of $X$. 
The basic strategy is the same as \cite[Proof of Theorem 1.2]{has-mmp}. 
We can assume that $\pi$ is a contraction and $K_{X}+\Delta$ is pseudo-effective over $Z$.

\begin{step}\label{step0.5}
In this step, we show that we may assume $X$ and $Z$ are projective. 

Let $Z\hookrightarrow Z^{c}$ be an open immersion to a normal projective variety $Z^{c}$. 
Thanks to \cite[Corollary 1.3]{has-mmp}, there is an lc closure $(X^{c},\Delta^{c})$ of $(X,\Delta)$, that is, a projective lc pair $(X^{c},\Delta^{c})$ such that $X$ is an open subset of $X^{c}$ and $(X^{c}|_{X},\Delta^{c}|_{X})=(X,\Delta)$, and there is a projective morphism $\pi^{c}\colon X^{c}\to Z^{c}$. 
By construction of lc closures, we have $\pi^{c}|_{X}=\pi$ and ${\pi^{c}}^{-1}(Z)=X$. 
Furthermore, we can construct $(X^{c},\Delta^{c})$ so that any lc center $S^{c}$ of $(X^{c},\Delta^{c})$ intersects $X$ (see \cite[Corollary 1.3]{has-mmp}). 
Then, the divisor $-(K_{X^{c}}+\Delta^{c})$ is pseudo-effective over $Z^{c}$ and for any lc center $S^{c}$ of $(X^{c},\Delta^{c})$, the pullback of $-(K_{X^{c}}+\Delta^{c})$ to the normalization of $S^{c}$ is pseudo-effective over $Z^{c}$ because relative numerical dimension of any $\mathbb{R}$-Cartier $\mathbb{R}$-divisor is determined on a sufficiently general fiber of the given morphism. 
Hence, we see that the morphism $(X^{c},\Delta^{c})\to Z^{c}$ satisfies the hypothesis of Theorem \ref{thmrelmmp}. 
If $(X^{c},\Delta^{c})$ has a good minimal model over $Z^{c}$, by restricting it over $Z$, we obtain a good minimal model of $(X,\Delta)$ over $Z$. 

In this way, by replacing $(X,\Delta)$ and $Z$ with $(X^{c},\Delta^{c})$ and $Z^{c}$, we may assume that $X$ and $Z$ are projective. 
\end{step}

\begin{step}\label{step1}
From this step to Step \ref{step5}, we prove that $(X,\Delta)$ has a log minimal model over $Z$. 
In this step, we construct a dlt blow-up with good properties. 
The strategy is the same as in \cite[Step 3 in the proof of Theorem 1.2]{has-mmp}. 

By the hypothesis, the relative numerical dimension of $K_{X}+\Delta$ over $Z$ is $0$. 
So there is $E\geq 0$ on $X$ such that $K_{X}+\Delta \sim_{\mathbb{R},Z}E$. 
Since $Z$ is projective, by adding the pullback of an ample divisor to $E$, we may assume that ${\rm Supp}E$ contains any lc center of $(X,\Delta)$ which is vertical over $Z$. 

We take a log resolution $f\colon \overline{X}\to X$ of $(X,{\rm Supp}(\Delta+E))$ and a log smooth model $(\overline{X},\overline{\Delta})$ of $(X,\Delta)$ (see \cite[Definition 2.9]{has-trivial} for definition of log smooth model). 
As in \cite[Proof of Lemma 2.10]{has-trivial}, by replacing $(\overline{X},\overline{\Delta})$ with a higher model, we may assume that we can write $\overline{\Delta}=\overline{\Delta}'+\overline{\Delta}''$ with $\overline{\Delta}'\geq0$ and $\overline{\Delta}''\geq0$ such that $\overline{\Delta}''$ is reduced and vertical over $Z$, and all lc centers of $(\overline{X},\overline{\Delta}')$ dominate $Z$. 
We can decompose $f^{*}E=\overline{G}+\overline{H}$ with $\overline{G}\geq0$ and $\overline{H}\geq0$ such that $\overline{G}$ and $\overline{H}$ have no common components, ${\rm Supp}\overline{G}\subset{\rm Supp}\llcorner \overline{\Delta}\lrcorner$ and no component of $\overline{H}$ is a component of $\llcorner \overline{\Delta}\lrcorner$. 
Since $(\overline{X},{\rm Supp}(\overline{\Delta}+\overline{H}))$ is log smooth, for any  $t>0$, if $ \overline{\Delta}+t\overline{H}$ is a boundary divisor then $(\overline{X},\overline{\Delta}+t\overline{H})$ is dlt. 
Since 
support of $E$ contains any lc center of $(X,\Delta)$ vertical over $Z$, we have ${\rm Supp}\overline{\Delta}''\subset{\rm Supp}\overline{G}$ . 
Moreover, since all lc centers of $(\overline{X},\overline{\Delta}')$ dominate $Z$, all lc centers of $(\overline{X},\overline{\Delta}-t\overline{G})$ dominate $Z$ for any $t>0$. 

We construct a dlt blow-up $(X_{0},\Delta_{0})\to (X,\Delta)$ by running the $(K_{\overline{X}}+\overline{\Delta})$-MMP over $X$. 
Let $G_{0}$ and $H_{0}$ be the birational transforms of $\overline{G}$ and $\overline{H}$ on $X_{0}$, respectively. 
By arguments of the log MMP, we can find $t_{0}>0$ such that for any $0<t\leq t_{0}$, the pair $(X_{0},\Delta_{0}+tH_{0})$ is dlt and all lc centers of $(X_{0},\Delta_{0}-tG_{0})$ dominate $Z$. 

In this way, by replacing $(X,\Delta)$, we can assume that $(X,\Delta)$ is $\mathbb{Q}$-factorial dlt and $K_{X}+\Delta\sim_{\mathbb{R},Z}G+H$ such that $G$ and $H$ satisfy
\begin{itemize}
\item
$G\geq0$, $H\geq0$, and $G$ and $H$ have no common components,
\item
${\rm Supp}G\subset {\rm Supp}\llcorner \Delta \lrcorner$, 
\item
any lc center of $(X,\Delta-tG)$ dominates $Z$ for any $0<t$, and 
\item
there is $t_{0}>0$ such that for any $0<t\leq t_{0}$, the pair $(X,\Delta+tH)$ is dlt. 
\end{itemize}
\end{step}

\begin{step}\label{step2}
Pick $0<\epsilon\leq t_{0}$ so that $\Delta-\epsilon G\geq0$, where $t_{0}$ is as in the fourth condition in Step \ref{step1}. 
In this step, we construct a strictly decreasing infinite sequence $\{e_{i}\}_{i\geq1}$ of real numbers and a sequence of birational maps over $Z$
$$X \dashrightarrow X_{1}\dashrightarrow X_{2}\dashrightarrow \cdots \dashrightarrow X_{i}\dashrightarrow \cdots$$
such that if we put $\Delta_{i}$ and $H_{i}$ as the birational transforms of $\Delta$ and $H$ on $X_{i}$ respectively, then
\begin{enumerate}
\item
$0<e_{i}<\epsilon$ and ${\rm lim}_{i\to \infty}e_{i}=0$, 
\item
$X \dashrightarrow X_{1}$ is a sequence of steps of the $(K_{X}+\Delta+e_{1}H)$-MMP over $Z$ to a good minimal model, 
\item
the sequence $X_{1}\dashrightarrow \cdots\dashrightarrow X_{i}\dashrightarrow \cdots$ is a sequence of finitely many steps of the $(K_{X_{1}}+\Delta_{1})$-MMP over $Z$ with scaling of $e_{1}H_{1}$, and 
\item
for any $i$, the pair $(X_{i},\Delta_{i}+e_{i}H_{i})$ is a good minimal model of $(X_{1},\Delta_{1}+e_{i}H_{1})$  over $Z$ and it is also a good minimal model of $(X,\Delta+e_{i}H)$ over $Z$. 
\end{enumerate}
By (4), if we set $\lambda_{i}={\rm inf}\{\mu \in \mathbb{R}_{\geq0}\,|\,K_{X_{i}}+\Delta_{i}+\mu H_{i} {\rm \; is \;nef\; over\;}Z\}$, then $\lambda_{i}\leq e_{i}$.  

Pick a strictly decreasing infinite sequence $\{e_{i}\}_{i\geq1}$ of positive real numbers such that $e_{i}<\epsilon$ for any $i\geq1$ and ${\rm lim}_{i\to \infty}e_{i}=0$. 
By conditions of Step \ref{step1}, the pairs $(X,\Delta+e_{i}H)$ and $(X,\Delta-\frac{e_{i}}{1+e_{i}}G)$ are dlt, and we have 
$$K_{X}+\Delta+e_{i}H\sim_{\mathbb{R},Z}(1+e_{i})\Bigl(K_{X}+\Delta-\frac{e_{i}}{1+e_{i}}G\Bigr).$$
Moreover, all lc centers of $(X,\Delta-\frac{e_{i}}{1+e_{i}}G)$ dominate $Z$ and the relative numerical dimension of $K_{X}+\Delta-\frac{e_{i}}{1+e_{i}}G$ over $Z$ is $0$ for any $i$. 
By \cite[Proposition 3.3]{has-mmp}, the pair $(X,\Delta-\frac{e_{i}}{1+e_{i}}G)$ has a good minimal model over $Z$, hence $(X,\Delta+e_{i}H)$ has a good minimal model over $Z$ for any $i$. 
By running the $(K_{X}+\Delta+e_{i}H)$-MMP over $Z$, we obtain a good minimal model
$(X,\Delta+e_{i}H)\dashrightarrow (X_{i},\Delta_{i}+e_{i}H_{i})$ over $Z$. 
Then, the log MMP only occurs in ${\rm Supp}(G+H)$, 
which does not depend on $i$. 
By replacing $\{e_{i}\}_{i\geq1}$ with a subsequence, we may assume that all birational maps $X \dashrightarrow X_{i}$ contract the same divisors, which implies that all $X_{i}$ are isomorphic in codimension one.

For any $0<t\leq e_{1}$, the pair $(X_{1},\Delta_{1}+tH_{1})$ has a good minimal model over $Z$. 
Indeed, we have 
$K_{X_{1}}+\Delta_{1}+tH_{1}\sim_{\mathbb{R},Z}(1+t)(K_{X_{1}}+\Delta_{1}-\frac{t}{1+t}G_{1})$ and the relative numerical dimension of $K_{X_{1}}+\Delta_{1}-\frac{t}{1+t}G_{1}$ over $Z$ is $0$, where $G_{1}$ is the birational transform of $G$ on $X_{1}$. 
Moreover, all lc centers of $(X_{1},\Delta_{1}-\frac{t}{1+t}G_{1})$ dominate $Z$. 
To check this, pick any prime divisor $P$ over $X_{1}$ such that $a(P,X_{1}, \Delta_{1}-\frac{t}{1+t}G_{1})=-1$. 
Since $(X_{1},\Delta_{1})$ is lc, we have $a(P,X_{1}, \Delta_{1}-\frac{e_{1}}{1+e_{1}}G_{1})=-1$. 
Since the birational map $X\dashrightarrow X_{1}$ is also a sequence of steps of the $(K_{X}+\Delta-\frac{e_{1}}{1+e_{1}}G)$-MMP, we have $a(P,X,\Delta-\frac{e_{1}}{1+e_{1}}G)=-1$. 
By the third condition in Step \ref{step1}, $P$ dominates $Z$. 
Thus, all lc centers of $(X_{1},\Delta_{1}-\frac{t}{1+t}G_{1})$ dominate $Z$. 
By \cite[Proposition 3.3]{has-mmp}, the pair $(X_{1},\Delta_{1}-\frac{t}{1+t}G_{1})$ has a good minimal model over $Z$, and so does $(X_{1},\Delta_{1}+tH_{1})$. 

Put $X'_{1}=X_{1}$ (resp.~$\Delta'_{1}=\Delta_{1}$, $H'_{1}=H_{1}$). 
By \cite[Lemma 2.14]{has-mmp}, we get a sequence of steps of the $(K_{X'_{1}}+\Delta'_{1})$-MMP over $Z$ with scaling of $e_{1}H'_{1}$
$$(X'_{1},\Delta'_{1})\dashrightarrow \cdots \dashrightarrow (X'_{j},\Delta'_{j})\dashrightarrow \cdots$$
such that if we set $\lambda'_{j}={\rm inf}\{\mu\in\mathbb{R}_{\geq0}\,|\,K_{X'_{j}}+\Delta'_{j}+\mu H'_{j}{\rm \; is\; nef\;over\;}Z\}$, where $H'_{j}$ is the birational transform of $H'_{1}$ on $X'_{j}$, then the $(K_{X'_{1}}+\Delta'_{1})$-MMP terminates after finitely many steps or we have ${\rm lim}_{j\to \infty}\lambda'_{j}=0$ when it does not terminate. 

For any $i\geq1$, pick the minimum $k_{i}$ such that $K_{X'_{k_{i}}}+\Delta'_{k_{i}}+e_{i}H'_{k_{i}}$ is nef over $Z$. 
Such $k_{i}$ exists since ${\rm lim}_{j\to \infty}\lambda'_{j}=0$, and we have $k_{1}=1$. 
By construction, the pair $(X'_{k_{i}},\Delta'_{k_{i}}+e_{i}H'_{k_{i}})$ is a good minimal model of $(X'_{1}, \Delta'_{1}+e_{i}H'_{1})$ over $Z$. 
We check that $(X'_{k_{i}},\Delta'_{k_{i}}+e_{i}H'_{k_{i}})$ is a good minimal model of $(X, \Delta+e_{i}H)$ over $Z$. 
Recall that for any $i$, the pair $(X_{i},\Delta_{i}+e_{i}H_{i})$ is a good minimal model of $(X,\Delta+e_{i}H)$ over $Z$, which was constructed at the start of this step, and all $X_{i}$ are isomorphic in codimension one. 
Since we put $X_{1}=X'_{1}$, $X'_{1}$ and $X_{i}$ are isomorphic in codimension one  for any $i$. 
Since ${\rm lim}_{i\to \infty}e_{i}=0$, the divisor $K_{X'_{1}}+\Delta'_{1}$ is the limit of movable divisors over $Z$. 
Then, the $(K_{X'_{1}}+\Delta'_{1})$-MMP contains only flips, and hence $X'_{k_{i}}$ and $X_{i}$ are isomorphic in codimension one. 
By Remark \ref{remmodels}, the pair $(X'_{k_{i}},\Delta'_{k_{i}}+e_{i}H'_{k_{i}})$ is a good minimal model of $(X,\Delta+e_{i}H)$ over $Z$. 

By abuse of notations, we put $X_{i}=X'_{k_{i}}$ (resp.~$\Delta_{i}=\Delta'_{k_{i}}$, $H_{i}=H'_{k_{i}}$) for any $i$. 
Note that after putting them, for any $i\geq 2$, the birational map $X\dashrightarrow X_{i}$ may not be a sequence of steps of the $(K_{X}+\Delta+e_{i}H)$-MMP.  
By construction, $\{e_{i}\}_{i\geq1}$ and 
$$X \dashrightarrow X_{1}\dashrightarrow X_{2}\dashrightarrow \cdots \dashrightarrow X_{i}\dashrightarrow \cdots$$
satisfy (1), (2), (3) and (4) stated at the start of this step. 
Indeed, (1) and (2) follow from the argument in the second paragraph.  
The conditions (3) and (4) follow from the arguments in the fourth paragraph and the fifth paragraph, respectively. 
\end{step}

\begin{step}\label{step3}
Suppose that the above $(K_{X_{1}}+\Delta_{1})$-MMP over $Z$ with scaling of $e_{1}H_{1}$ terminates. 
Then $X_{l}\simeq X_{l+1}\simeq \cdots$ for some $l$, and hence, for any $i\geq l$, the pair $(X_{l},\Delta_{l}+e_{i}H_{l})$ is a good minimal model of $(X,\Delta+e_{i}H)$ over $Z$ by (4) in Step \ref{step2}. 
Then, we have $a(P,X,\Delta+e_{i}H)\leq a(P,X_{l},\Delta_{l}+e_{i}H_{l})$ for any prime divisor $P$ over $X$. 
By considering the limit $i\to \infty$, we have $a(P,X,\Delta)\leq a(P,X_{l},\Delta_{l})$. 
So the pair $(X_{l},\Delta_{l})$ is a weak lc model of $(X,\Delta)$ over $Z$, and thus, we see that $(X,\Delta)$ has a log minimal model over $Z$. 

Therefore, to show the existence of log minimal model of $(X,\Delta)$ over $Z$, we only have to prove the termination of the $(K_{X_{1}}+\Delta_{1})$-MMP.
\end{step}

\begin{step}\label{step4}
Since we have $K_{X_{1}}+\Delta_{1}+e_{1}H_{1}\sim_{\mathbb{R},Z}(1+e_{1})(K_{X_{1}}+\Delta_{1}-\frac{e_{1}}{1+e_{1}}G_{1})$ and ${\rm Supp}G_{1}\subset{\rm Supp}\llcorner \Delta_{1}\lrcorner$, the $(K_{X_{1}}+\Delta_{1})$-MMP only occurs in ${\rm Supp}\llcorner \Delta_{1}\lrcorner$ (see, for example, \cite[Step 2 in the proof of Proposition 5.4]{has-trivial}). 

Suppose that the $(K_{X_{1}}+\Delta_{1})$-MMP does not terminate. 
We get a contradiction by the argument of the special termination (\cite{fujino-sp-ter}). 
We note that $(X_{1},\Delta_{1}+e_{1}H_{1})$ is $\mathbb{Q}$-factorial dlt and any lc center of the pair is an lc center of $(X_{1},\Delta_{1})$. 
Therefore, for any $i$, the pair $(X_{i},\Delta_{i})$ is $\mathbb{Q}$-factorial dlt and any lc center of it is normal. 
There is $m>0$ such that for any lc center $S_{m}$ of $(X_{m},\Delta_{m})$ and  any $i \geq m$, the indeterminacy locus of the birational map $X_{m}\dashrightarrow X_{i}$ does not contain $S_{m}$ and the restriction of the map to $S_{m}$ induces a birational map. 
For any lc center $S_{m}$ of $(X_{m},\Delta_{m})$, let $S_{i}$ be the lc center of $(X_{i},\Delta_{i})$ birational to $S_{m}$, and we define $\Delta_{S_{i}}$ on $S_{i}$ by adjunction $K_{S_{i}}+\Delta_{S_{i}}=(K_{X_{i}}+\Delta_{i})|_{S_{i}}$. 
In this step and the next step, we prove that for any $S_{m}$, there is $i_{0}\geq m$ such that the induced birational map $(S_{i}, \Delta_{S_{i}})\dashrightarrow (S_{i+1}, \Delta_{S_{i+1}})$ is an isomorphism for any $i\geq i_{0}$. 
If we can prove this, the $(K_{X_{1}}+\Delta_{1})$-MMP must terminate (see \cite{fujino-sp-ter}), and we get a contradiction. 

We prove the assertion by induction on the dimension of $S_{m}$. 
Let $\Upsilon_{m}\subset S_{m}$ be an lc center of $(X_{m},\Delta_{m})$. 
As in \cite{fujino-sp-ter}, by replacing $m$, we may assume that for any $i\geq m$, if $\Upsilon_{m}\subsetneq S_{m}$ then the map $(\Upsilon_{m}, \Delta_{\Upsilon_{m}})\dashrightarrow (\Upsilon_{i}, \Delta_{\Upsilon_{i}})$ is an isomorphism. 
Moreover, as in \cite{fujino-sp-ter}, by replacing $m$ again, we may assume that if $\Upsilon_{m}= S_{m}$ then 
the map $\Upsilon_{m}\dashrightarrow \Upsilon_{i}$ is small  and the birational transform of $\Delta_{\Upsilon_{m}}$ on $\Upsilon_{i}$ is $\Delta_{\Upsilon_{i}}$.
Let $(T_{m},\Psi_{m})\to (S_{m},\Delta_{S_{m}})$ be a dlt blow-up. 
We set $H_{T_{m}}$ as the pullback of $H_{m}|_{S_{m}}$ to $T_{m}$. 
By the argument as in \cite{fujino-sp-ter} (see also \cite[Remark 2.10]{birkar-flip}), we obtain a diagram 
$$
\xymatrix
{
(T_{m},\Psi_{m}) \ar[d] \ar@{-->}[r]&
\cdots  \ar@{-->}[r]& (T_{i},\Psi_{i}) \ar[d]\ar@{-->}[r]&\cdots \\
(S_{m},\Delta_{S_{m}})\ar@{-->}[r]&\cdots \ar@{-->}[r]&(S_{i},\Delta_{S_{i}})\ar@{-->}[r]&\cdots
}
$$
such that 
\begin{itemize}
\item
$(T_{i},\Psi_{i})\to(S_{i},\Delta_{S_{i}})$ is a dlt blow-up, and
\item
the upper horizontal sequence of birational maps is a sequence of steps of the $(K_{T_{m}}+\Psi_{m})$-MMP over $Z$ with scaling of $e_{m}H_{T_{m}}$. 
\end{itemize}
We prove that the $(K_{T_{m}}+\Psi_{m})$-MMP over $Z$ must terminate. 
If we can prove this, then we can find $i_{0}\geq m$ such that the induced map $(S_{i}, \Delta_{S_{i}})\dashrightarrow (S_{i+1}, \Delta_{S_{i+1}})$ is an isomorphism for any $i\geq i_{0}$. 
By \cite[Theorem 4.1 (iii)]{birkar-flip}, to prove the termination of the $(K_{T_{m}}+\Psi_{m})$-MMP, it is sufficient to prove that $(T_{m},\Psi_{m})$ has a log minimal model over $Z$. 
Since the morphism $(T_{m},\Psi_{m})\to (S_{m},\Delta_{S_{m}})$ is a dlt blow-up, it is sufficient to prove that $(S_{m},\Delta_{S_{m}})$ has a log minimal model over $Z$. 
\end{step}

\begin{step}\label{step5}
We prove that $(S_{m},\Delta_{S_{m}})$ has a log minimal model over $Z$ by using the induction hypothesis of Theorem \ref{thmrelmmp}. 
Since $(X_{m},\Delta_{m})$ is $\mathbb{Q}$-factorial dlt, $S_{m}$ and all lc centers of $(S_{m},\Delta_{S_{m}})$ are lc centers of $(X_{m},\Delta_{m})$ contained in $S_{m}$. 
Since the divisors $K_{S_{i}}+\Delta_{S_{i}}+e_{i}H_{i}|_{S_{i}}$ are nef over $Z$ and since the map $S_{m}\dashrightarrow S_{i}$ is small, 
by recalling ${\rm lim}_{i\to \infty}e_{i}=0$, we see that 
$K_{S_{m}}+\Delta_{S_{m}}$ is pseudo-effective over $Z$. 
From these facts, it is sufficient to check that $-(K_{X_{m}}+\Delta_{m})|_{\Upsilon_{m}}$ is pseudo-effective over $Z$ for any lc center $\Upsilon_{m}\subset S_{m}$ of $(X_{m},\Delta_{m})$.  

Recall that for any $i\geq m$ and any lc center $\Upsilon_{m}\subset S_{m}$ of $(X_{m},\Delta_{m})$, the induced map $\Upsilon_{m}\dashrightarrow \Upsilon_{i}$ is in particular small and the birational transform of $\Delta_{\Upsilon_{m}}$ on $\Upsilon_{i}$ is $\Delta_{\Upsilon_{i}}$
We put $H_{\Upsilon_{i}}=H_{i}|_{\Upsilon_{i}}$. 
Then $H_{\Upsilon_{i}}\geq0$ and the birational transform of $H_{\Upsilon_{m}}$ on $\Upsilon_{i}$ is $H_{\Upsilon_{i}}$. 
By construction of the map $(X,\Delta+e_{i}H)\dashrightarrow(X_{i},\Delta_{i}+e_{i}H_{i})$ (see (2) and (3) in Step \ref{step2}), there is an lc center $\Upsilon$ of $(X,\Delta)$ such that the birational map $X\dashrightarrow X_{i}$ induces a birational map $\Upsilon\dashrightarrow \Upsilon_{i}$. 
We put $H_{\Upsilon}=H|_{\Upsilon}$, and we define $\Delta_{\Upsilon}$ on $\Upsilon$ by adjunction $K_{\Upsilon}+\Delta_{\Upsilon}=(K_{X}+\Delta)|_{\Upsilon}$. 
Then $H_{\Upsilon}\geq0$. 
By (2) and (3) in Step \ref{step2}, for any $i\geq m$, there is a common log resolution $Y_{i}\to X$ and $Y_{i}\to {X_{i}}$ of $X\dashrightarrow X_{i}$ and a subvariety $\Upsilon_{Y_{i}}\subset Y_{i}$ birational to $\Upsilon$ and $\Upsilon_{i}$ such that the induced morphisms $\Upsilon_{Y_{i}}\to \Upsilon$ and $\Upsilon_{Y_{i}}\to \Upsilon_{i}$ form a common resolution of the map $\Upsilon\dashrightarrow \Upsilon_{i}$. 
Using (4) in Step \ref{step2} and the negativity lemma, by taking pullbacks of $K_{X}+\Delta+e_{i}H$ and $K_{X_{i}}+\Delta_{i}+e_{i}H_{i}$ to $\Upsilon_{Y_{i}}$ and comparing coefficients, we see that $a(Q,\Upsilon,\Delta_{\Upsilon}+e_{i}H_{\Upsilon})\leq a(Q,\Upsilon_{i},\Delta_{\Upsilon_{i}}+e_{i}H_{\Upsilon_{i}})$ for any prime divisor $Q$ over $\Upsilon$. 

Since $(X_{m},\Delta_{m})$ is $\mathbb{Q}$-factorial dlt, the pair $(\Upsilon_{m},\Delta_{\Upsilon_{m}})$ is dlt. 
So there is a small $\mathbb{Q}$-factorialization $\Upsilon'\to \Upsilon_{m}$. 
Then, $\Upsilon'$ and $\Upsilon_{i}$ are isomorphic in codimension one for any $i\geq m$ because $\Upsilon_{m}$ and $\Upsilon_{i}$ are isomorphic in codimension one. 
We denote the pullback of $K_{\Upsilon_{m}}+\Delta_{\Upsilon_{m}}$ to $\Upsilon'$ by $K_{\Upsilon'}+\Delta_{\Upsilon'}$. 
We take a common resolution $\varphi\colon\overline{\Upsilon}\to \Upsilon$ and $\varphi'\colon\overline{\Upsilon}\to \Upsilon'$ of the birational map $\Upsilon\dashrightarrow \Upsilon'$. 
For any $i\geq m$, we take a common resolution $\tau\colon\overline{\Upsilon}_{i}\to \overline{\Upsilon}$ and $\tau_{i}\colon\overline{\Upsilon}_{i}\to \Upsilon_{i}$ of the birational map $\overline{\Upsilon}\dashrightarrow \Upsilon_{i}$. 
We have the following diagram.
$$
\xymatrix@R=12pt
{
&\overline{\Upsilon}\ar[ddl]_{\varphi}\ar[d]^{\varphi'}&\overline{\Upsilon}_{i}\ar[l]_{\tau}\ar[ddr]^{\tau_{i}}\\
&\Upsilon' \ar[d]\\
\Upsilon\ar@{-->}[r]&\Upsilon_{m}\ar@{-->}[r]&\cdots \ar@{-->}[r]&\Upsilon_{i}\ar@{-->}[r]&\cdots
}
$$
Since we have $a(Q,\Upsilon,\Delta_{\Upsilon}+e_{i}H_{\Upsilon})\leq a(Q,\Upsilon_{i},\Delta_{\Upsilon_{i}}+e_{i}H_{\Upsilon_{i}})$ for any prime divisor $Q$ over $\Upsilon$, we have 
$$\tau^{*}\varphi^{*}(K_{\Upsilon}+\Delta_{\Upsilon}+e_{i}H_{\Upsilon})-\tau_{i}^{*}(K_{\Upsilon_{i}}+\Delta_{\Upsilon_{i}}+e_{i}H_{\Upsilon_{i}})\geq0.$$
Therefore, 
$$-\tau_{i}^{*}(K_{\Upsilon_{i}}+\Delta_{\Upsilon_{i}})+e_{i}\tau^{*}\varphi^{*}H_{\Upsilon}\geq-\tau^{*}\varphi^{*}(K_{\Upsilon}+\Delta_{\Upsilon})+e_{i}\tau_{i}^{*}H_{\Upsilon_{i}}.$$
By the hypothesis of Theorem \ref{thmrelmmp}, $-(K_{\Upsilon}+\Delta_{\Upsilon})$ is pseudo-effective over $Z$. 
Thus, the divisor $-\tau_{i}^{*}(K_{\Upsilon_{i}}+\Delta_{\Upsilon_{i}})+e_{i}\tau^{*}\varphi^{*}H_{\Upsilon}$ is pseudo-effective over $Z$ since $H_{\Upsilon_{i}}\geq0$. 
We have $\varphi'_{*}\tau_{*}\tau_{i}^{*}(K_{\Upsilon_{i}}+\Delta_{\Upsilon_{i}})=K_{\Upsilon'}+\Delta_{\Upsilon'}$ since $\Upsilon'$ and $\Upsilon_{i}$ are isomorphic in codimension one. 
By taking the birational transform on $\Upsilon'$, we see that the divisor 
$-(K_{\Upsilon'}+\Delta_{\Upsilon'})+e_{i}\varphi'_{*}\varphi^{*}H_{\Upsilon}$
 is pseudo-effective over $Z$ for any $i$. 
Note that $\Upsilon'$ is $\mathbb{Q}$-factorial. 
Since ${\rm lim}_{i\to \infty}e_{i}=0$, the divisor $-(K_{\Upsilon'}+\Delta_{\Upsilon'})$ is pseudo-effective over $Z$. 
So we see that $-(K_{\Upsilon_{m}}+\Delta_{\Upsilon_{m}})$ is pseudo-effective over $Z$. 

In this way, the restriction $-(K_{X_{m}}+\Delta_{m})|_{\Upsilon_{m}}$ is pseudo-effective over $Z$ for any lc center $\Upsilon_{m}\subset S_{m}$ of $(X_{m},\Delta_{m})$. 
By the induction hypothesis of Theorem \ref{thmrelmmp}, $(S_{m},\Delta_{S_{m}})$ has a log minimal model over $Z$. 
Therefore, we can find $i_{0}\geq m$ such that the induced birational map $(S_{i}, \Delta_{S_{i}})\dashrightarrow (S_{i+1}, \Delta_{S_{i+1}})$ is an isomorphism for any $i\geq i_{0}$ (see Step \ref{step4}). 
Then, by the argument of the special termination (\cite{fujino-sp-ter}), the $(K_{X_{1}}+\Delta_{1})$-MMP over $Z$ must terminate. 

In this way, we see that $(X,\Delta)$ has a log minimal model (see Step \ref{step3}). 
\end{step}

\begin{step}\label{step6}
By running the $(K_{X}+\Delta)$-MMP over $Z$, we can obtain a log minimal model $(X,\Delta)\dashrightarrow(X_{\rm min},\Delta_{\rm min})$ over $Z$. 
Then, the numerical dimension of $K_{X_{\rm min}}+\Delta_{\rm min}$ over $Z$ is $0$, and for any lc center $S'$ of $(X_{\rm min},\Delta_{\rm min})$, the numerical dimension of $(K_{X_{\rm min}}+\Delta_{\rm min})|_{S'}$ over $Z$ is $0$. 
Since $X$ and $Z$ are both projective, we can apply Lemma \ref{lemneflogabund}. 
Therefore, the divisor $K_{X_{\rm min}}+\Delta_{\rm min}$ is semi-ample, and $(X_{\rm min},\Delta_{\rm min})$ is a good minimal model over $Z$. 
\end{step}
So we are done. 
\end{proof}

The following result is not used in this paper, but it is interesting on its own. 

\begin{cor}
Let $\pi\colon X\to Z$ be a projective morphism of normal quasi-projective varieties, and let $(X,\Delta)$ be an lc pair. 
Suppose that there is an $\mathbb{R}$-divisor $B\geq0$ on $X$ such that
\begin{itemize}
\item
$-(K_{X}+\Delta+B)$ is nef over $Z$, and 
\item 
$(X,\Delta+\epsilon B)$ is lc for a real number $\epsilon>0$.  
\end{itemize}
Then, $(X,\Delta)$ has a good minimal model or a Mori fiber space over $Z$. 
\end{cor}

\begin{proof}
We can check that the morphism $(X,\Delta)\to Z$ satisfies the hypothesis of Theorem \ref{thmrelmmp}. 
Therefore, the corollary follows from Theorem \ref{thmrelmmp}. 
\end{proof}

\section{Pseudo-lc pairs}\label{sec4}

In this section, a pair $\langle X,\Delta \rangle$ simply denotes a pair of a normal variety $X$ and an $\mathbb{R}$-divisor $\Delta\geq0$ on it. 
In particular, we do not assume $K_{X}+\Delta$ to be $\mathbb{R}$-Cartier. 
When $K_{X}+\Delta$ is $\mathbb{R}$-Cartier, we denote the pair of $X$ and $\Delta$ by $(X,\Delta)$ as usual. 

\begin{defn}\label{defnalmostdiscrepancy}
Let $\langle X,\Delta \rangle$ be a pair and let $P$ be a prime divisor over $X$, that is, a prime divisor on a higher birational model $Y\to X$. 
We define the {\em discrepancy} $\alpha(P,X,\Delta)$ of $P$ with respect to $\langle X,\Delta \rangle$ as follows: 

We fix $K_{X}$ as a Weil divisor. 
We denote the image of $P$ on $X$ by $c_{X}(P)$. 
Let $f\colon Y\to X$ be a projective birational morphism from a normal variety $Y$ such that $P$ is a prime divisor on $Y$. 
We fix $K_{Y}$ so that $f_{*}K_{Y}=K_{X}$ as Weil divisors. 
The divisor $K_{Y}$ depends on the choice of $K_{X}$. 
For any affine open subset $U\subset X$ such that $U\cap c_{X}(P)\neq \emptyset$, we put $K_{U}=K_{X}|_{U}$, $V=f^{-1}(U)$, $f_{V}=f|_{V}$ and $K_{V}=K_{Y}|_{V}$. 
For any $\mathbb{R}$-divisor $B_{U}\geq0$ on $U$ such that $K_{U}+\Delta|_{U}+B_{U}$ is $\mathbb{R}$-Cartier, we define 
$$\alpha_{\langle X,\Delta \rangle}(P,U,B_{U})={\rm coeff}_{P|_{V}}\bigl(K_{V}-f_{V}^{*}(K_{U}+\Delta|_{U}+B_{U})\bigr).$$
By the standard argument, $\alpha_{\langle X,\Delta \rangle}(P,U,B_{U})$ does not depend on the choice of $K_{X}$ and $f\colon Y\to X$. 
We define
$$\alpha(P,X,\Delta):=\underset{U,B_{U}}{\rm sup}\{\alpha_{\langle X,\Delta \rangle}(P,U,B_{U})\},$$
where $U$ runs over all affine open subsets of $X$ such that $U\cap c_{X}(P)\neq \emptyset$, and $B_{U}$ runs over all effective $\mathbb{R}$-divisors on $U$ such that $K_{U}+\Delta|_{U}+B_{U}$ is $\mathbb{R}$-Cartier. 
\end{defn}

\begin{defn}\label{defnalmostpair}
Let $\langle X,\Delta \rangle$ be a pair. 
We say the pair $\langle X,\Delta \rangle$ is {\em pseudo-lc} if the inequality $\alpha(P,X,\Delta)\geq -1$ holds for any prime divisor $P$ over $X$. 
\end{defn}

We show three basic properties of discrepancy defined above.

\begin{lem}\label{lembasic}
Let $\langle X,\Delta \rangle$ be a pair and $P$ be a prime divisor over $X$. 
\begin{enumerate}
\item[(i)]
If $K_{X}+\Delta$ is $\mathbb{R}$-Cartier, then $\alpha(P,X,\Delta)=a(P,X,\Delta)$, where the right hand side is the usual discrepancy. 
\item[(ii)]
If  $P$ is a divisor on $X$, then $\alpha(P,X,\Delta)=-{\rm coeff}_{P}(\Delta)$. 
\item[(iii)]
Let $0\leq\Delta'\leq\Delta$ be an $\mathbb{R}$-divisor. 
Then $\alpha(P,X,\Delta)\leq \alpha(P,X,\Delta')$.
\end{enumerate}
In particular, if $\langle X,\Delta \rangle$ is pseudo-lc and $K_{X}+\Delta$ is $\mathbb{R}$-Cartier, then $(X,\Delta)$ is lc. 
\end{lem}

\begin{proof}
These are proved by the standard arguments. 

Firstly, we prove (i). The inequality $\alpha(P,X,\Delta)\geq a(P,X,\Delta)$ follows from the definition of $\alpha(P,X,\Delta)$. 
So we prove the inverse inequality.  
Let $f\colon Y\to X$ be a projective birational morphism such that $P$ is a prime divisor on $Y$. Let $U$ be an affine open subset of $X$ such that $U\cap c_{X}(P)\neq \emptyset$. We set $V=f^{-1}(U)$ and $f_{V}=f|_{V}$. 
For any $\mathbb{R}$-Cartier $\mathbb{R}$-divisor $B_{U} \geq 0$ on $U$,  $a(P,X,\Delta)-\alpha_{\langle X,\Delta \rangle}(P,U,B_{U})$ is the coefficient of $P|_{V}$ in 
$$(K_{V}-f_{V}^{*}(K_{U}+\Delta|_{U}))-(K_{V}-f_{V}^{*}(K_{U}+\Delta|_{U}+B_{U}))\geq0. $$ Hence, we have $a(P,X,\Delta)\geq\alpha_{\langle X,\Delta \rangle}(P,U,B_{U})$ for any $U$ and $B_{U}$. By taking the supremum, we have $a(P,X,\Delta)\geq \alpha(P,X,\Delta)$. 
So the equality holds. 

Secondly, we show (ii). For any affine open subset $U\subset X$ with $P\cap U\neq \emptyset$ and any $\mathbb{R}$-divisor $B_{U}\geq0$ on $U$ such that $K_{U}+\Delta|_{U}+B_{U}$ is $\mathbb{R}$-Cartier, we have $\alpha_{\langle X,\Delta \rangle}(P,U,B_{U})\leq -{\rm coeff}_{P}(\Delta)$. 
Then $\alpha(P,X,\Delta)\leq-{\rm coeff}_{P}(\Delta)$ by Definition \ref{defnalmostdiscrepancy}. 
We pick an affine open subset $U$ such that $P\cap U\neq \emptyset$ and $U$ is contained in the smooth locus of $X$. 
Such $U$ exists since $X$ is normal. 
Then, the divisor $K_{U}+\Delta|_{U}$ is $\mathbb{R}$-Cartier, and we have $\alpha_{\langle X,\Delta \rangle}(P,U,0)=-{\rm coeff}_{P}(\Delta)$. 
By Definition \ref{defnalmostdiscrepancy}, we have $\alpha(P,X,\Delta)\geq-{\rm coeff}_{P}(\Delta)$. 
Thus, the equality of (ii) holds. 

Finally, we show (iii). 
Put $G=\Delta-\Delta'\geq0$, and pick any prime divisor $P$ over $X$. 
For any affine open subset $U\subset X$ with $U\cap c_{X}(P)\neq \emptyset$ and any $\mathbb{R}$-divisor $B_{U}\geq 0$ on $U$ such that $K_{U}+\Delta|_{U}+B_{U}$ is $\mathbb{R}$-Cartier, we have $$\alpha_{\langle X,\Delta\rangle}(P,U,B_{U})=\alpha_{\langle X,\Delta'\rangle}(P,U,G|_{U}+B_{U})\leq \alpha(P,X,\Delta')$$ by Definition \ref{defnalmostdiscrepancy}. 
By taking the supremum, we have $\alpha(P,X,\Delta)\leq \alpha(P,X,\Delta')$. 
So we are done. 
\end{proof}

By Lemma \ref{lembasic}, we see that the usual lc pairs and potentially lc pairs are pseudo-lc pairs (for definition of potentially lc pairs, see \cite[Definition 17]{kollar-logpluri}). 
We will see later that the notion of pseudo-lc singularity is closely related to log canonical singularity introduced in \cite{dfh} and generalized lc pairs introduced in  \cite{bz} (Proposition \ref{propdfh} and Proposition \ref{propgeneralizedlc}). 

Before showing results on pseudo-lc pairs, we recall notations and definitions in \cite{dfh}. 
In \cite{dfh}, de Fernex and Hacon defined log canonical and log terminal singularities for pairs $(X, \sum a_{i}Z_{i})$ of a normal quasi-projective variety $X$ and a formal $\mathbb{R}_{\geq 0}$-linear combination 
$\sum a_{i}Z_{i}$ of subschemes $Z_{i}\subset X$. 
In this paper, we deal with the case when the subscheme part is zero. 
We note that the definition of log canonical singularity in the sense of \cite{dfh} (Definition \ref{defndfh}) is only used for comparison to pseudo-lc singularity. 

\begin{note}\label{note4.4}
For any prime divisor $P$ over $X$, we denote by $v_{P}\colon \mathbb{C}(X)\to \mathbb{Z}$ the corresponding divisorial valuation on the field of rational functions $\mathbb{C}(X)$. 

Firstly, for any Weil divisor $D$ on $X$, we define 
\begin{equation}
v_{P}^{\natural}(D):={\rm min}\left\{ v_{P}(\phi)
 \left|\begin{array}{l}\phi \in \mathcal{O}_{X}( -D)(U), U\subset X {\rm \;is\;open},\\U\cap c_{X}(P)\neq \emptyset.\end{array}\right.\right\}. 
 \end{equation}
(\cite[Definition 2.1 and Definition 2.2]{dfh}).

Secondly, for any birational morphism $f\colon Y \to X$ from a normal variety $Y$, we define
\begin{equation}
f^{\natural}D:=\sum_{\substack {E:{\rm \,prime\,divisor}\\{\rm on\,}Y}}v_{E}^{\natural}(D)E
\end{equation}
(\cite[Definition 2.6]{dfh}).

Thirdly, for any $\mathbb{Q}$-divisor $D$, we can define $v_{P}(D)$ by
\begin{equation}
v_{P}(D):=\underset{k\geq 1}{\rm inf}\frac{v_{P}^{\natural}(kD)}{k}=\underset{k\to \infty}{\rm lim}{\rm inf}\frac{v_{P}^{\natural}(kD)}{k}=\underset{k\to \infty}{\rm lim}\frac{v_{P}^{\natural}(k!D)}{k!}
\end{equation}
(\cite[Lemma 2.8 and Definition 2.9]{dfh}). 
When $D$ is $\mathbb{Q}$-Cartier, $v_{P}(D)$ coincides with the usual valuation along $P$.

Fourthly, we fix a birational morphism $f\colon Y \to X$ from a normal variety $Y$ and Weil divisors $K_{X}$ and $K_{Y}$ such that $f_{*}K_{Y}=K_{X}$. 
For any $m\geq 1$, we put
\begin{equation}
K_{m, Y/X}:=K_{Y}-\frac{1}{m}f^{\natural}(m K_{X})
\end{equation}
(\cite[Definition 3.1]{dfh}). 
Note that $K_{m, Y/X}$ does not depend on the choice of $K_{X}$ and $K_{Y}$.

Finally, for any $\mathbb{Q}$-divisor $D$ on $X$ and any birational morphism $f\colon Y \to X$ from a normal variety $Y$, the pullback of $D$ is defined by
\begin{equation}f^{*}D:=\sum_{\substack {E:{\rm \,prime\,divisor}\\{\rm on\,}Y}}v_{E}(D)E,\end{equation}
where $v_{E}(D)$ is as in Notation \ref{note4.4} (3). 
If $D$ is $\mathbb{Q}$-Cartier, $f^{*}D$ coincides with the usual pullback.
\end{note}

\begin{defn}[{\cite[Definition 7.1]{dfh}}]\label{defndfh}
Let $X$ be a normal variety. 
We say the pair $\langle X,0 \rangle$ is {\em log canonical} {\em in the sense of} {\cite{dfh}} if $X$ is quasi-projective and there is $m\geq 1$ such that 
for any projective birational morphism $f\colon Y \to X$ and any prime divisor $P$ on $Y$, coefficient of $P$ in $K_{m, Y/X}$ is not less than $-1$. 
\end{defn}

By \cite[Proposition 7.2]{dfh}, $\langle X,0 \rangle$ is log canonical in the sense of \cite{dfh} if and only if there is a  boundary $\mathbb{Q}$-divisor $\Delta$ such that $K_{X}+\Delta$ is $\mathbb{Q}$-Cartier and $(X,\Delta)$ is lc. 
Therefore, by Lemma \ref{lembasic} (i) and (iii), if $X$ is log canonical in the sense of \cite{dfh} then $\langle X,0 \rangle$ is pseudo-lc. 

Moreover, we also have the following statement: 

\begin{prop}\label{propdfh}
Let $\langle X,\Delta=\sum d_{i}\Delta_{i} \rangle$ be a pair, where $\Delta_{i}$ are (not necessarily prime or effective) Weil divisors. 
If there is $m\geq1$ such that 
$${\rm coeff}_{P}(K_{m, Y/X})-\sum d_{i}\cdot v_{P}^{\natural}(\Delta_{i})\geq-1$$
for any projective birational morphism $f\colon Y \to X$ and any prime divisor $P$ on $Y$,  
then $\langle X,\Delta \rangle$ is pseudo-lc. 
\end{prop}

\begin{proof}
We pick any $P$ over $X$ and fix $f\colon Y \to X$ such that $P$ is a prime divisor on $Y$. 
By Notation \ref{note4.4} (1), there is an open subset $U_{0}\subset X$ and a rational function $\phi_{0}$ such that $U_{0}\cap c_{X}(P)\neq \emptyset$, $\phi_{0} \in \mathcal{O}_{X}(-mK_{X})(U_{0})$ and $v_{P}(\phi_{0})=v_{P}^{\natural}(mK_{X})$. 
Similarly, for any $i$, we can find $U_{i}$ and $\phi_{i}$ such that $U_{i}\cap c_{X}(P)\neq \emptyset$, $\phi_{i} \in \mathcal{O}_{X}(-\Delta_{i})(U_{i})$ and $v_{P}(\phi_{i})=v_{P}^{\natural}(\Delta_{i})$. 
By shrinking $U_{0}$ and $U_{i}$, we may assume that $U_{0}=U_{i}$ for any $i$ and $U_{0}$ is affine. 
We put $U=U_{0}$. 
We define divisors $B_{0}$ and $B_{i}$ on $U$ by
$$B_{0}:=\frac{1}{m}({\rm div}(\phi_{0})-mK_{U})\quad {\rm and}\quad B_{i}:={\rm div}(\phi_{i})-\Delta_{i}|_{U}.$$
By construction of $\phi_{0}$ and $\phi_{i}$, we have $B_{0}\geq0$ and $B_{i}\geq0$. 
Set $B_{U}=B_{0}+\sum d_{i}B_{i}$. 
Then, the divisor $K_{U}+\Delta|_{U}+B_{U}$ is $\mathbb{R}$-Cartier because 
\begin{equation*}
\begin{split}
K_{U}+\Delta|_{U}+B_{U}&=(K_{U}+B_{0})+\sum d_{i}(\Delta_{i}|_{U}+B_{i})\\
&=\frac{1}{m}{\rm div}(\phi_{0})+\sum d_{i}\cdot{\rm div}(\phi_{i}).
\end{split}
\end{equation*}
Moreover, if we set $V=f^{-1}(U)$ and $f_{V}=f|_{V}$, we have
\begin{equation*}
\begin{split}
\alpha_{\langle X,\Delta \rangle}(P,U,B_{U})=&{\rm coeff}_{P|_{V}}\bigl(K_{V}-f_{V}^{*}(K_{U}+\Delta|_{U}+B_{U})\bigr)\\
=&{\rm coeff}_{P|_{V}}\bigl(K_{V}-(\frac{1}{m}f_{V}^{*}({\rm div}(\phi_{0}))+\sum d_{i}f_{V}^{*}({\rm div}(\phi_{i})))\bigr)\\
=&{\rm coeff}_{P}(K_{Y})-\frac{1}{m}v_{P}(\phi_{0})-\sum d_{i}\cdot v_{P}(\phi_{i}).
\end{split}
\end{equation*}
We recall that $v_{P}(\phi_{0})=v_{P}^{\natural}(mK_{X})$ and $v_{P}(\phi_{i})=v_{P}^{\natural}(\Delta_{i})$. 
With the above equation, we obtain 
\begin{equation*}
\begin{split}
\alpha_{\langle X,\Delta \rangle}(P,U,B_{U})=&{\rm coeff}_{P}(K_{Y})-\frac{1}{m}v_{P}(\phi_{0})-\sum d_{i}\cdot v_{P}(\phi_{i})\\
=&{\rm coeff}_{P}(K_{Y})-\frac{1}{m}v_{P}^{\natural}(mK_{X})-\sum d_{i}\cdot v_{P}^{\natural}(\Delta_{i})\\
\geq&-1.
\end{split}
\end{equation*}
By Definition \ref{defnalmostdiscrepancy}, we have $\alpha(P,X,\Delta)\geq-1$ for any prime divisor $P$ over $X$. 
In this way, we see that $\langle X,\Delta\rangle$ is pseudo-lc. 
\end{proof}

We discuss other formulations of discrepancy in Definition \ref{defnalmostdiscrepancy} (Proposition \ref{propdiscrepancy}, Theorem \ref{proppairdiscrepancy} and Theorem \ref{thmnefenvelope}). 
In Proposition \ref{propdiscrepancy} below, we give a very simple description of discrepancy with Notation \ref{note4.4} (5).

\begin{prop}\label{propdiscrepancy}
Let $\langle X,\Delta \rangle$ be a pair such that $\Delta$ is a $\mathbb{Q}$-divisor, and let $P$ be a prime divisor over $X$. 
Let $f\colon Y\to X$ be a projective birational morphism from a normal variety $Y$ such that $P$ is a divisor on $Y$. 
Fix $K_{X}$ and $K_{Y}$ such that $f_{*}K_{Y}=K_{X}$. 
Then,
$$\alpha(P,X,\Delta)={\rm coeff}_{P}\bigl(K_{Y}-f^{*}(K_{X}+\Delta)\bigr)={\rm coeff}_{P}(K_{Y})-v_{P}(K_{X}+\Delta).$$
\end{prop}

\begin{proof}
The second equality is obvious from the definition of $f^{*}(K_{X}+\Delta)$. 
We prove the equality $\alpha(P,X,\Delta)={\rm coeff}_{P}(K_{Y})-v_{P}(K_{X}+\Delta).$
Pick any $m$ such that $m=k!$ for some integer $k>0$ and $m\Delta$ is a Weil divisor. 
By Notation \ref{note4.4} (1),  we can find an open subset $U \subset X$ and a rational function $\phi\in \mathcal{O}_{X}(-m(K_{X}+\Delta))(U)$ such that $U\cap c_{X}(P)\neq \emptyset$ and $v_{P}(\phi)=v_{P}^{\natural}(m(K_{X}+\Delta))$. 
By shrinking $U$, we may assume that $U$ is affine. 
We set $B_{U}=\frac{1}{m}({\rm div}(\phi)-m(K_{U}+\Delta|_{U}))$. 
Then $B_{U}\geq0$ and the divisor $K_{U}+\Delta|_{U}+B_{U}$ is $\mathbb{Q}$-Cartier. 
If we put $V=f^{-1}(U)$ and $f_{V}=f|_{V}$, then 
\begin{equation*}
\begin{split}
\alpha(P,X,\Delta)\geq\alpha_{\langle X,\Delta \rangle}(P,U,B_{U})=&{\rm coeff}_{P|_{V}}\bigl(K_{V}-f_{V}^{*}(K_{U}+\Delta|_{U}+B_{U})\bigr)\\
=&{\rm coeff}_{P}(K_{Y})-\frac{1}{m}v_{P}(\phi)\\
=&{\rm coeff}_{P}(K_{Y})-\frac{1}{m}v_{P}^{\natural}(m(K_{X}+\Delta)). 
\end{split}
\end{equation*}
By Notation \ref{note4.4} (3), we have $v_{P}(K_{X}+\Delta)={\rm lim}_{k\to \infty}\frac{v_{P}^{\natural}(k!(K_{X}+\Delta))}{k!}$. 
Therefore, considering the limit $k\to \infty$, we obtain $\alpha(P,X,\Delta)\geq{\rm coeff}_{P}(K_{Y})-v_{P}(K_{X}+\Delta)$. 

On the other hand, pick an affine open subset $U'\subset X$ and an $\mathbb{R}$-divisor $C_{U'}\geq 0$ on $U'$ such that $U'\cap c_{X}(P)\neq \emptyset$ and $K_{U'}+\Delta|_{U'}+C_{U'}$ is $\mathbb{R}$-Cartier. 
Then, there are positive real numbers $r_{1},\cdots,r_{n}$ and effective $\mathbb{Q}$-divisors $C_{1},\cdots,C_{n}$ on $U'$ such that $\sum_{j=1}^{n}r_{j}=1$, $\sum_{j=1}^{n}r_{j}C_{j}=C_{U'}$ and $K_{U'}+\Delta|_{U'}+C_{j}$ are $\mathbb{Q}$-Cartier. 
By definition of $\alpha_{\langle X,\Delta \rangle}(P,U',C_{U'})$, we have 
$\alpha_{\langle X,\Delta \rangle}(P,U',C_{U'})=\sum_{j=1}^{n}r_{j}\cdot\alpha_{\langle X,\Delta \rangle}(P,U',C_{j}).$
Therefore, we have  $\alpha_{\langle X,\Delta \rangle}(P,U',C_{U'})\leq \alpha_{\langle X,\Delta \rangle}(P,U',C_{j'})$ for some index $j'$. 
Pick a sufficiently large and divisible integer $m>0$ such that $m\Delta$ and $mC_{j'}$ are both Weil divisors and $m(K_{U'}+\Delta|_{U'}+C_{j'})$ is Cartier. 
By shrinking $U'$, we may write $m(K_{U'}+\Delta|_{U'}+C_{j'})={\rm div}(\sigma)$ with a rational function $\sigma$. 
Since $C_{j'}\geq 0$, we have $\sigma \in \mathcal{O}_{X}(-m(K_{X}+\Delta))(U')$, and therefore we obtain $v_{P}(\sigma)\geq v_{P}^{\natural}(m(K_{X}+\Delta))$ by Notation \ref{note4.4} (1). 
With Notation \ref{note4.4} (3), we have $$\frac{1}{m}v_{P}(\sigma)\geq\frac{1}{m}v_{P}^{\natural}(m(K_{X}+\Delta))\geq v_{P}(K_{X}+\Delta).$$ 
We put $V'=f^{-1}(U')$ and $f_{V'}=f|_{V'}$. 
From the above facts, for any $U'$ and $C_{U'}$, we have 
\begin{equation*}
\begin{split}
\alpha_{\langle X,\Delta \rangle}(P,U',C_{U'})\leq &\alpha_{\langle X,\Delta \rangle}(P,U',C_{j'})\\=&{\rm coeff}_{P|_{V'}}\bigl(K_{V'}-f_{V'}^{*}(K_{U'}+\Delta|_{U'}+C_{j'})\bigr)\\
=&{\rm coeff}_{P}(K_{Y})-\frac{1}{m}v_{P}(\sigma)\\
\leq&{\rm coeff}_{P}(K_{Y})-v_{P}(K_{X}+\Delta).
\end{split}
\end{equation*}
By taking the supremum, we have $\alpha(P,X,\Delta)\leq {\rm coeff}_{P}(K_{Y})-v_{P}(K_{X}+\Delta)$. 
So we obtain the desired equality. 
\end{proof}

Next, we prove that for any pair $\langle X,\Delta \rangle$ such that $X$ is quasi-projective, the discrepancy $\alpha(\,\cdot\,,X,\Delta)$ of $\langle X,\Delta \rangle$ can be approximated by the usual discrepancy of pairs $(X,\Delta+G)$ with $G\geq 0$. 

\begin{thm}\label{proppairdiscrepancy}
Let $\langle X, \Delta \rangle$ be a pair such that $X$ is quasi-projective. 
Then, for any projective birational morphism $f\colon Y\to X$ from a normal quasi-projective variety $Y$ and any real number $\epsilon>0$, there is an effective $\mathbb{R}$-divisor $G$ on $X$ such that 
\begin{itemize} \item $\Delta$ and $G$ have no common components, and 
\item $K_{X}+\Delta+G$ is $\mathbb{R}$-Cartier and $\alpha(P,X,\Delta)-a(P,X,\Delta+G)\leq\epsilon$ for any prime divisor $P$ on $Y$, where $a(P,X,\Delta+G)$ is the usual discrepancy. \end{itemize}
In particular, for any prime divisor $P$ over $X$, we have 
$$\alpha(P,X,\Delta)={\rm sup}\{a(P,X,\Delta+G)\,|\,G\geq0{\rm \; such\;that\;}K_{X}+\Delta+G {\rm \;is \;}\mathbb{R}{\rm \mathchar`-Cartier}\}.$$
\end{thm}

\begin{proof}
The second assertion immediately follows from the first assertion. 
So we only prove the first assertion. 
Pick $f\colon Y\to X$ and $\epsilon>0$ as in Theorem \ref{proppairdiscrepancy}. 
By replacing $Y$ by a higher smooth model, we may assume that $Y$ is smooth. 
Fix Weil divisors $K_{X}$ and $K_{Y}$ such that $f_{*}K_{Y}=K_{X}$. 
We prove Theorem \ref{proppairdiscrepancy} in two steps. 
\begin{step3}\label{step1discre}
First we prove Theorem \ref{proppairdiscrepancy} when $\Delta$ is a $\mathbb{Q}$-divisor. 
We borrow the idea of \cite[Proof of Theorem 5.4]{dfh}. 

Let $\{E_{i}\}_{i}$ be the set of all $f$-exceptional prime divisors on $Y$. 
Since the set $\{E_{i}\}_{i}$ is a finite set, by Notation \ref{note4.4} (3), there is a sufficiently large and divisible integer $m>0$ such that $m\Delta$ is a Weil divisor and $\frac{1}{m}v_{E_{i}}^{\natural}(m(K_{X}+\Delta))\leq v_{E_{i}}(K_{X}+\Delta)+\epsilon$ for all $E_{i}$. 
We pick $m>0$ such that $\frac{1}{m}\leq \epsilon$ and $m$ satisfies the above condition. 
By Proposition \ref{propdiscrepancy}, we have
\begin{equation*}
\begin{split}
\alpha(E_{i},X,\Delta)=&{\rm coeff}(K_{Y})-v_{E_{i}}(K_{X}+\Delta)\\
\leq& {\rm coeff}(K_{Y})-\frac{1}{m}v_{E_{i}}^{\natural}(m(K_{X}+\Delta))+\epsilon.
\end{split}
\end{equation*}
Pick a Weil divisor $D\geq0$ on $X$ such that $m(K_{X}+\Delta)-D$ is Cartier, and take an ample Cartier divisor $A$ such that the sheaf $\mathcal{O}_{X}(A-D)$ is globally generated. 
We can find such $D$ and $A$ since $X$ is quasi-projective. 
By construction of $\mathcal{O}_{X}(A-D)$, we have
$${\rm min}\{v_{P}(\psi)|\,\psi \in H^{0}(X,\mathcal{O}_{X}(A-D))\}=v_{P}^{\natural}(D-A)$$
for any prime divisor $P$ on $Y$, where $v_{P}^{\natural}(\,\cdot\,)$ is as in Notation \ref{note4.4} (1).   
We define a linear system
$$|A-D|=\{A'\in |A|\;|\,A'-D\geq0\}=\{{\rm div}(\psi)+A\,|\,\psi \in H^{0}(X,\mathcal{O}_{X}(A-D))\}$$
and consider its pullback $f^{*}|A-D|:=\{f^{*}A'\,|\,A'\in |A-D|\}$. 
Then, the fixed part ${\rm Fix}(f^{*}|A-D|)$ is 
\begin{equation*}
\begin{split}
{\rm Fix}(f^{*}|A-D|)=&\sum_{P}
\bigl({\rm min}\{{\rm coeff}_{P}(f^{*}A')\,|\,A'\in |A-D|\,\}\bigr)P\\
=&\sum_{P}
\left({\rm min}\{v_{P}(\psi)|\psi \in H^{0}(X,\mathcal{O}_{X}(A-D))\}+{\rm coeff}_{P}(f^{*}A)\right)P\\
=&f^{*}A+\sum_{P} v_{P}^{\natural}(D-A)\cdot P=f^{*}A+f^{\natural}(D-A)\\
=&f^{\natural}D, 
\end{split}
\end{equation*}
where the final equality follows from \cite[Lemma 2.4]{dfh}. 
Therefore, we can find a movable Cartier divisor $M$ such that $M+f^{\natural}D\sim f^{*}A$. 
Then, $f_{*}M+D$ is Cartier. 
Thus, the divisor $m(K_{X}+\Delta)+f_{*}M=m(K_{X}+\Delta)-D+(D+f_{*}M)$ is Cartier and we have $M+f^{\natural}D=f^{*}(f_{*}M+D)$.
We pick $M\geq0$ so that $M$ is reduced and it contains no $f$-exceptional divisors or components of $f_{*}^{-1}\Delta$ in its support. 
Then 
\begin{equation*}
\begin{split}
&K_{Y}-\frac{1}{m}M-\frac{1}{m}f^{\natural}(m(K_{X}+\Delta))\\
=&K_{Y}-\frac{1}{m}M-\frac{1}{m}f^{\natural}(m(K_{X}+\Delta)-D+D)\\
=&K_{Y}-\frac{1}{m}M-\frac{1}{m}f^{\natural}D-\frac{1}{m}f^{*}(m(K_{X}+\Delta)-D)\\
=&K_{Y}-\frac{1}{m}f^{*}(f_{*}M+D)-\frac{1}{m}f^{*}(m(K_{X}+\Delta)-D)\\
=&K_{Y}-f^{*}(K_{X}+\Delta+\frac{1}{m}f_{*}M),
\end{split}
\end{equation*}
where the second equality follows from \cite[Lemma 2.4]{dfh} and that the divisor $m(K_{X}+\Delta)-D$ is Cartier. 
We recall that $m$ satisfies $\frac{1}{m}\leq \epsilon$, and also recall that we have $\alpha(E_{i},X,\Delta)\leq {\rm coeff}_{E_{i}}(K_{Y})-\frac{1}{m}v_{E_{i}}^{\natural}(m(K_{X}+\Delta))+\epsilon$ for any $f$-exceptional prime divisor $E_{i}$ on $Y$. 
Pick any prime divisor $P$ on $Y$. 
Since $M$ contains no $f$-exceptional divisors, if $P$ is $f$-exceptional,  we have 
\begin{equation*}
\begin{split}
a(P,X,\Delta+\frac{1}{m}f_{*}M)=&{\rm coeff}_{P}\bigl(K_{Y}-\frac{1}{m}M-\frac{1}{m}f^{\natural}(m(K_{X}+\Delta))\bigr)\\
=&{\rm coeff}_{P}(K_{Y})-\frac{1}{m}v_{P}^{\natural}(m(K_{X}+\Delta))\\
\geq&\alpha(P,X,\Delta)-\epsilon.
\end{split}
\end{equation*}
If $P$ is a divisor on $X$, we have 
\begin{equation*}
\begin{split}
\alpha(P,X,\Delta)-a(P,X,\Delta+\frac{1}{m}f_{*}M)=&\frac{1}{m}\cdot{\rm coeff}_{P}(f_{*}M)
\leq\frac{1}{m}\leq \epsilon, 
\end{split}
\end{equation*}
where the first equality follows from Lemma \ref{lembasic} (ii) and the second inequality follows from that $M$ is reduced. 
So $\frac{1}{m}f_{*}M$ satisfies the conditions of Theorem \ref{proppairdiscrepancy}. 
\end{step3}

\begin{step3}\label{step2discre}
From now on, we prove Theorem \ref{proppairdiscrepancy} when $\Delta$ is an $\mathbb{R}$-divisor. 

Let $\{E_{i}\}_{i}$ be the set of all $f$-exceptional prime divisors on $Y$. 
By Definition \ref{defnalmostdiscrepancy}, there are affine open subsets $U_{i}\subset X$ with $c_{X}(E_{i})\cap U_{i}\neq \emptyset$ and $\mathbb{R}$-divisors $B_{i}\geq0$ on $U_{i}$ such that $K_{U_{i}}+\Delta|_{U_{i}}+B_{i}$ are $\mathbb{R}$-Cartier and $\alpha(E_{i},X,\Delta)-\frac{\epsilon}{3}\leq \alpha_{\langle X,\Delta \rangle}(E_{i},U_{i},B_{i})$ for all $i$. 
Let $\mathcal{E}\subset{\rm WDiv}_{\mathbb{R}}(X)$ be the set of effective $\mathbb{R}$-divisors on $X$ whose support is contained in ${\rm Supp}\Delta$. 
For any $i$, let $\mathcal{B}_{i}\subset{\rm WDiv}_{\mathbb{R}}(U_{i})$ be the set of effective $\mathbb{R}$-divisors on $U_{i}$ whose support is contained in ${\rm Supp}B_{i}$. 
We identify $\mathcal{E}$ (resp.~$\mathcal{B}_{i}$) with a subset of the $\mathbb{R}$-vector space whose basis is given by all components of $\Delta$ (resp.~the $\mathbb{R}$-vector space whose basis is given by all components of $B_{i}$). 
Consider the set 
$$\Bigl\{\bigl(\Delta',(B'_{i})_{i}\bigr)\in \mathcal{E}\times \underset{i}{\prod}\mathcal{B}_{i}\Bigm|K_{U_{i}}+\Delta'|_{U_{i}}+B'_{i} {\rm \;\,is \;\, }\mathbb{R}{\rm \mathchar`-Cartier\;\, for\;\, any\;\,}i \Bigr\}$$
which contains $\bigl(\Delta,(B_{i})_{i}\bigr)$. 
By an argument of convex geometry, we see that the set contains a rational polytope in $\mathcal{E}\times \prod_{i}\mathcal{B}_{i}$ containing 
$\bigl(\Delta,(B_{i})_{i}\bigr)$. 
Therefore, we can find positive real numbers $r_{1},\cdots, r_{n}$, effective $\mathbb{Q}$-divisors $\Delta^{(1)}, \cdots ,\Delta^{(n)}$ on $X$ and effective $\mathbb{Q}$-divisors $B_{i}^{(1)}, \cdots ,B_{i}^{(n)}$ on $U_{i}$ such that $\sum_{l=1}^{n}r_{l}=1$, $\sum_{l=1}^{n}r_{l}\Delta^{(l)}=\Delta$, $\sum_{l=1}^{n}r_{l}B_{i}^{(l)}=B_{i}$ and $K_{U_{i}}+\Delta^{(l)}|_{U_{i}}+B_{i}^{(l)}$ is $\mathbb{Q}$-Cartier for any $i$. 
By choosing those $\mathbb{Q}$-divisors sufficiently close to $\Delta$ and $B_{i}$, we may assume that the inequality $ \alpha_{\langle X,\Delta \rangle}(E_{i},U_{i},B_{i})-\frac{\epsilon}{3}\leq \alpha_{\langle X,\Delta^{(l)} \rangle}(E_{i},U_{i},B_{i}^{(l)})$ holds for any $i$ and $l$. 
Then
\begin{equation*} \begin{split} 
\alpha(E_{i},X,\Delta^{(l)}) \geq&\alpha_{\langle X,\Delta^{(l)} \rangle}(E_{i},U_{i},B_{i}^{(l)})
\geq
 \alpha_{\langle X,\Delta \rangle}(E_{i},U_{i},B_{i})-\frac{\epsilon}{3}\\
\geq&\alpha(E_{i},X,\Delta)-\frac{2}{3}\epsilon,
\end{split} \end{equation*}
where the first inequality follows from Definition \ref{defnalmostdiscrepancy}. 
Furthermore, we can assume that ${\rm Supp}\Delta={\rm Supp}\Delta^{(l)}$ and all coefficients of $\Delta-\Delta^{(l)}$ belong to $[-\frac{2}{3}\epsilon, \frac{2}{3}\epsilon]$ for any $1\leq l \leq n$. 
Then, by Lemma \ref{lembasic} (ii) and the above inequality, we obtain 
$$\alpha(P,X,\Delta)-\alpha(P,X,\Delta^{(l)})\leq \frac{2}{3}\epsilon$$
for any prime divisor $P$ on $Y$. 
By the case of $\mathbb{Q}$-divisors of Theorem \ref{proppairdiscrepancy}, we can find effective $\mathbb{R}$-divisors $G^{(1)},\cdots,G^{(n)}$ on $X$ such that 
\begin{itemize}
\item
$\Delta^{(l)}$ and $G^{(l)}$ have no common components, and 
\item
$K_{X}+\Delta^{(l)}+G^{(l)}$ is $\mathbb{R}$-Cartier and $\alpha(P,X,\Delta^{(l)})-a(P,X,\Delta^{(l)}+G^{(l)})\leq\frac{\epsilon}{3}$ for any prime divisor $P$ on $Y$
\end{itemize}
for any $1\leq l \leq n$. 
We set $G=\sum_{l=1}^{n}r_{l}G^{(l)}$. 
By construction, we have 
$$K_{X}+\Delta+G=\sum_{l=1}^{n}r_{l}(K_{X}+\Delta^{(l)}+G^{(l)}),$$
and so $K_{X}+\Delta+G$ is $\mathbb{R}$-Cartier. 
Since ${\rm Supp}\Delta={\rm Supp}\Delta^{(l)}$ for any $1\leq l \leq n$ and since $\Delta^{(l)}$ and $G^{(l)}$ have no common components, we see that $\Delta$ and $G$ have no common components. 
We pick any prime divisor $P$ on $Y$. 
By construction, we have $a(P,X,\Delta+G)=\sum_{l=1}^{n}r_{l}\cdot a(P,X,\Delta^{(l)}+G^{(l)})$. 
Recalling $\sum_{l=1}^{n}r_{l}=1$, we obtain
\begin{equation*}
\begin{split}
&\alpha(P,X,\Delta)-a(P,X,\Delta+G)\\
=&\sum_{l=1}^{n}r_{l}\bigl(\alpha(P,X,\Delta)-a(P,X,\Delta^{(l)}+G^{(l)})\bigr)\\
=&\sum_{l=1}^{n}r_{l}\bigl(\alpha(P,X,\Delta)-\alpha(P,X,\Delta^{(l)})+\alpha(P,X,\Delta^{(l)})-a(P,X,\Delta^{(l)}+G^{(l)})\bigr)\\
\leq&\sum_{l=1}^{n}r_{l}\left(\frac{2}{3}\epsilon+\frac{1}{3}\epsilon\right)=\epsilon.
\end{split}
\end{equation*}
In this way, $G$ satisfies the conditions of Theorem \ref{proppairdiscrepancy}. 
\end{step3}
So we are done. 
\end{proof}

We also see that the $b$-divisor defined with discrepancies is a logarithmic analog of the relative log canonical $b$-divisor in \cite[Definition 3.1]{bdff}.  

\begin{thm}\label{thmnefenvelope}
Let $\langle X,\Delta \rangle$ be a pair, and let $f\colon Y\to X$ be a projective birational morphism from a normal variety $Y$. 
Put $D=\sum_{P}\alpha(P,X,\Delta)$, where $P$ runs over all prime divisors on $Y$. 

Then, we have $D=K_{Y}+({\rm Env}_{X}(-(K_{X}+\Delta)))_{Y},$
where $({\rm Env}_{X}(-(K_{X}+\Delta)))_{Y}$ is the trace of the  nef envelope ${\rm Env}_{X}(-(K_{X}+\Delta))$ on $Y$ $($for the definition of nef envelope, see \cite[Definition 2.3]{bdff}$)$. 
\end{thm}

\begin{proof}
The inequality  
$D\geq K_{Y}+({\rm Env}_{X}(-(K_{X}+\Delta)))_{Y}$ follows from Definition \ref{defnalmostdiscrepancy} and \cite[Definition 2.3]{bdff} (see also \cite[Lemma 2.2]{bdff}), and the equality holds when $\Delta$ is a $\mathbb{Q}$-divisor (Proposition \ref{propdiscrepancy} and \cite[Remark 2.4]{bdff}). 
Furthermore, by the same argument as in Step \ref{step2discre} in the proof of Theorem \ref{proppairdiscrepancy}, for any $\epsilon >0$, we can find positive real numbers $r_{1},\cdots, r_{n}$ and $\mathbb{Q}$-divisors $\Delta^{(1)}, \cdots ,\Delta^{(n)}$ such that $\sum_{l=1}^{n}r_{l}=1$, $\sum_{l=1}^{n}r_{l}\Delta^{(l)}=\Delta$ and $\alpha(P,X,\Delta)-\alpha(P,X,\Delta^{(l)})\leq \epsilon$ for any $l$ and any prime divisor $P$ on $Y$. 
Then $\alpha(P,X,\Delta)-\sum_{l=1}^{n}r_{l}\alpha(P,X,\Delta^{(l)})\leq \epsilon$, and therefore 
\begin{equation*}
\begin{split}
&{\rm coeff}_{P}\bigl(K_{Y}+({\rm Env}_{X}(-(K_{X}+\Delta)))_{Y}\bigr)\\
\geq &\sum_{l=1}^{n}r_{l}\cdot{\rm coeff}_{P}\bigl(K_{Y}+({\rm Env}_{X}(-(K_{X}+\Delta^{(l)})))_{Y}\bigr)=\sum_{l=1}^{n}r_{l}\alpha(P,X,\Delta^{(l)})\\
\geq& \alpha(P,X,\Delta)-\epsilon={\rm coeff}_{P}(D)-\epsilon, 
\end{split}
\end{equation*}
where the first inequality follows from \cite[Proposition 2.6]{bdff}, and the second equality follows because $\Delta^{(l)}$ are $\mathbb{Q}$-divisors. 
Since $\epsilon$ is any positive real number, we have $D\leq K_{Y}+({\rm Env}_{X}(-(K_{X}+\Delta)))_{Y}$. 
So the equality holds.  
\end{proof}

We give two examples of pseudo-lc pairs. 
First one is pseudo-lc pairs $\langle Z,\Delta_{Z}\rangle$ which are not lc. 

\begin{exam}\label{examnotlc}
Let $(X,\Delta)$ be a projective $\mathbb{Q}$-factorial klt pair such that the Picard number $\rho(X)$ is greater than $1$ and $-(K_{X}+\Delta)$ is nef but not numerically trivial. 
We pick a very ample Cartier divisor $A$ on $X$ such that there is no real number $r$ satisfying $rA\sim_{\mathbb{R}}K_{X}+\Delta$. 
Note that we only use $\rho(X)>1$ for the existence of $A$. 
Set $Y=\mathbb{P}_{X}(\mathcal{O}_{X}\oplus \mathcal{O}_{X}(-A))$, and let $f\colon Y \to X$ be the natural morphism. 
Then 
$$K_{Y}+2S+f^{*}\Delta+f^{*}A=f^{*}(K_{X}+\Delta),$$ where $S$ is the unique section corresponding to $\mathcal{O}_{Y}(1)$. 
We note that $S$ is Cartier, $S\simeq X$ and the pair $(Y,S+f^{*}\Delta)$ is plt. 
We construct a cone $Z$ by contracting $S$. 
Let $\pi \colon Y\to Z$ be the natural morphism. 
By construction, the image $\pi(S)$ is a point. 
Moreover, we can write $S+f^{*}A\sim_{\mathbb{Q}}\pi^{*}H$ for an ample $\mathbb{Q}$-divisor $H$ on $Z$. 
We put $\Delta_{Z}=\pi_{*}f^{*}\Delta$. 

We show that $\langle Z,\Delta_{Z} \rangle$ is pseudo-lc. 
For any real number $t>0$, pick a general ample $\mathbb{R}$-divisor $A_{t}\sim_{\mathbb{R}}tA-(K_{X}+\Delta)$. 
Since we have $K_{X}+\Delta+A_{t}\sim_{\mathbb{R}}tA$, we see that $K_{Z}+\Delta_{Z}+\pi_{*}f^{*}A_{t}$ is $\mathbb{R}$-Cartier (\cite[Proposition 7.2.8]{fujino-book}). 
Then, by a simple calculation, we obtain
$$K_{Y}+f^{*}A_{t}+f^{*}\Delta+(1+t)S=\pi^{*}(K_{Z}+\Delta_{Z}+\pi_{*}f^{*}A_{t}).$$ 
Let $P$ be any prime divisor over $Z$. 
By replacing $A_{t}$ if necessary, we may assume $c_{Y}(P)\not\subset {\rm Supp}f^{*}A_{t}$. 
Then, we have $a(P,Z,\Delta_{Z}+\pi_{*}f^{*}A_{t})=a(P,Y,(1+t)S+f^{*}\Delta)$, where both hand sides are the usual discrepancies. 
By definition of $\alpha(P,Z,\Delta_{Z})$ (see Definition \ref{defnalmostdiscrepancy}), we have $\alpha(P,Z,\Delta_{Z})\geq a(P,Z,\Delta_{Z}+\pi_{*}f^{*}A_{t})$ for any $t>0$. 
Thus, we obtain $\alpha(P,Z,\Delta_{Z})\geq a(P,Y,(1+t)S+f^{*}\Delta)$ for any $t>0$. 
By the standard argument of discrepancies and since the pair $(Y,S+f^{*}\Delta)$ is plt, the function $\mathbb{R}\ni t' \mapsto a(P,Y,(1+t')S+f^{*}\Delta)$ is continuous and $a(P,Y,S+f^{*}\Delta)\geq -1$. 
Since we have $\alpha(P,Z,\Delta_{Z})\geq a(P,Y,(1+t)S+f^{*}\Delta)$ for any $t>0$, by considering the limit $t\to 0$, we obtain $\alpha(P,Z,\Delta_{Z})\geq -1$. 
Thus, we see that $\langle Z,\Delta_{Z}\rangle$ is pseudo-lc. 

We show that $\langle Z,\Delta_{Z}\rangle$ is not lc. 
It is sufficient to show that $K_{Z}+\Delta_{Z}$ is not $\mathbb{R}$-Cartier. 
Recall that there is no real number $r$ such that $rA\sim_{\mathbb{R}}K_{X}+\Delta$. 
Then, $K_{Z}+\Delta_{Z}$ is not $\mathbb{R}$-Cartier by \cite[Proposition 7.2.8]{fujino-book}. 
Thus, $\langle Z,\Delta_{Z}\rangle$ is not lc. 
\end{exam}

Next example is pseudo-lc pairs which are not log canonical in the sense of \cite{dfh}.

\begin{exam}[see also {\cite[Theorem 1.3]{yzhang}}]\label{examnotdfhlc}
Let $X$ be a normal projective variety such that $(X,0)$ is $\mathbb{Q}$-factorial klt, $-K_{X}$ is nef and there is no effective $\mathbb{Q}$-divisor $\Delta \sim_{\mathbb{Q}}-K_{X}$ such that $(X,\Delta)$ is lc. 
Such variety $X$ exists even if $X$ is a smooth surface (\cite[Example 1.1]{shokurovcompl}). 
As in Example \ref{examnotlc}, we pick a very ample divisor $A$ on $X$ and set $Y=\mathbb{P}_{X}(\mathcal{O}_{X}\oplus \mathcal{O}_{X}(-A))$. 
Note that there is no real number $r$ such that $K_{X}\sim_{\mathbb{R}}rA$ by the assumption on $K_{X}$. 
Let $f\colon Y \to X$ be the natural morphism and $\pi \colon Y\to Z$ be the contraction of the section $S$ corresponding to $\mathcal{O}_{Y}(1)$. 
We have $K_{Y}+2S+f^{*}A=f^{*}K_{X}$ and $S+f^{*}A\sim_{\mathbb{Q}} \pi^{*}H$ for an ample $H$ on $Z$. 
We also have $S \simeq X$, and $\pi(S)$ is a point. 

Since $-K_{X}$ is nef, as in the argument in the second paragraph of Example \ref{examnotlc}, we see that $\langle Z,0 \rangle$ is pseudo-lc. 
We show that $\langle Z,0\rangle$ is not log canonical in the sense of \cite{dfh}. 
If $\langle Z,0\rangle$ is log canonical in the sense of \cite{dfh}, by \cite[Proposition 7.2]{dfh}, there is a $\mathbb{Q}$-divisor $B\geq0$ on $Z$ such that $K_{Z}+B$ is $\mathbb{Q}$-Cartier and $(Z,B)$ is lc. 
Then, we can write $K_{Y}+aS+\pi_{*}^{-1}B=\pi^{*}(K_{Z}+B)$ with an $a\leq 1$, and the pair $(Y,aS+\pi_{*}^{-1}B)$ is sub-lc. 
If $a<1$, by using $S+f^{*}A\sim_{\mathbb{Q}}\pi^{*}H$, we obtain
$$K_{Y}+S+\pi_{*}^{-1}B+(1-a)f^{*}A\sim_{\mathbb{Q}}\pi^{*}(K_{Z}+B+(1-a)H).$$
By restricting to $S$, we obtain $K_{S}\sim_{\mathbb{R}}-\pi_{*}^{-1}B|_{S}-(1-a)f^{*}A|_{S}$. 
We recall $S\simeq X$. 
Since $\pi_{*}^{-1}B|_{S}\geq 0$ and $1-a>0$, we see that $-K_{X}$ is big. 
Because $-K_{X}$ is nef and $(X,0)$ is $\mathbb{Q}$-factorial klt by the hypothesis, we can find a $\mathbb{Q}$-divisor $\Delta\sim_{\mathbb{Q}}-K_{X}$ such that $(X,\Delta)$ is klt. 
But it contradicts the hypothesis of $X$. 
Thus, we see that $a=1$. 
Then $K_{Y}+S+\pi_{*}^{-1}B=\pi^{*}(K_{Z}+B)$ and the pair $(Y,S+\pi_{*}^{-1}B)$ is lc. 
By restricting to $S$, we obtain $K_{S}\sim_{\mathbb{Q}}-\pi_{*}^{-1}B|_{S}$, and if we set $\Delta_{S}=\pi_{*}^{-1}B|_{S}$, then $\Delta_{S}$ is a $\mathbb{Q}$-divisor and the pair $(S,\Delta_{S})$ is lc by adjunction. 
Since $S\simeq X$, there is an $\mathbb{Q}$-divisor $\Delta_{X}\sim_{\mathbb{Q}}-K_{X}$ such that $(X,\Delta_{X})$ is lc. 
But it contradicts the hypothesis of $X$. 
Therefore, $\langle Z,0\rangle$ is not log canonical in the sense of \cite{dfh}. 
\end{exam}

The following proposition says that pseudo-lc pairs appear in generalized lc pairs. 
For definition of generalized lc pairs, see \cite[Definition 4.1]{bz}. 

\begin{prop}\label{propgeneralizedlc}
Let $(X',\Delta'+M')$ be a generalized lc pair which comes with a data $X\to X'\to Z$ and $M$. 
Then, the pair $\langle X',\Delta' \rangle$ is pseudo-lc. 
\end{prop}

\begin{proof}
By definition of pseudo-lc pairs, we can shrink $X'$ and $Z$. 
Therefore, we may assume that $Z$ is affine and there is an ample divisor on $X'$. 
We fix a prime divisor $P$ over $X'$, and we show $\alpha(P,X',\Delta')\geq-1$. 
We denote $X\to X'$ by $f$. 
By replacing $X$, we may assume that $f$ is a log resolution of $\langle X',{\rm Supp}\Delta'\rangle$ such that $P$ is a divisor on $X$. 
We can write $K_{X}+\Delta+M=f^{*}(K_{X'}+\Delta'+M')$, where $(X,\Delta)$ is sub-lc. 
Pick an ample divisor $A'$ on $X'$ and write $f^{*}A'\sim_{\mathbb{R}}H+G$, where $H$ is ample and $G\geq0$. 
For any $t>0$, we pick a general member $H_{t}\sim_{\mathbb{R}}tH+M$ such that $H_{t}\geq0$ and ${\rm Supp}H_{t}\nsupseteq P$. 
Then, we have $K_{X'}+\Delta'+f_{*}(H_{t}+tG)\sim_{\mathbb{R}}K_{X'}+\Delta'+M'+tA'$ and so $K_{X'}+\Delta'+f_{*}(H_{t}+tG)$ is $\mathbb{R}$-Cartier. 
We also have $f_{*}(H_{t}+tG)\geq0$ and
$$K_{X}+\Delta+H_{t}+tG=f^{*}(K_{X'}+\Delta'+f_{*}(H_{t}+tG))$$ for any $t>0$. 
Since $(X,\Delta)$ is sub-lc, by definition of $\alpha(P,X',\Delta')$, we have 
$$\alpha(P,X',\Delta')\geq {\rm coeff}_{P}(-\Delta-tG)\geq -1-t\cdot{\rm coeff}_{P}(G)$$
for any $t>0$. 
So $\alpha(P,X',\Delta')\geq -1$, and we see that $\langle X',\Delta' \rangle$ is pseudo-lc. 
\end{proof}

\begin{rem}\label{remdiffpseudolc}
We give two remarks on Example \ref{examnotdfhlc}. 
\begin{itemize}
\item[(1)]
Example \ref{examnotdfhlc} shows that there is a generalized lc pair with zero boundary part $(Z, M_{Z})$ such that there is no divisor $B$ with which the pair $(Z,B)$ is lc. 
Indeed, with notation as in Example \ref{examnotdfhlc}, put $N=-f^{*}K_{X}+\pi^{*}H$, which is nef by construction of $X$. 
Then we have 
$K_{Y}+S+N\sim_{\mathbb{Q}}0$.
Since $(Y,S)$ is plt, $(Z,M_{Z}:=g_{*}N)$ is a generalized lc pair which comes with the data $\pi\colon Y\to Z$ and $N$. 
But, as we have seen in Example \ref{examnotdfhlc}, there is no boundary divisor $B$ such that the pair $(Z,B)$ is lc (\cite[Proposition 7.2]{dfh}). 
\item[(2)]
Example \ref{examnotdfhlc} gives a negative answer to question (b) in \cite[Section 0]{bdff}. 
Indeed, with notation as in Example \ref{examnotdfhlc}, take $X$ as a smooth surface as in \cite[Example 1.1]{shokurovcompl}. 
Then $Z$ has only one isolated singular point $z_{0}=\pi(S)$. 
Since $\langle Z,0 \rangle$ is pseudo-lc and by Theorem \ref{thmnefenvelope}, we see that the log discrepancy $b$-divisor as in \cite[Definition 3.4]{bdff} is effective. 
Therefore, if ${\rm Vol}(Z,z_{0})$ is the volume defined in \cite[Definition 4.18]{bdff}, then we have ${\rm Vol}(Z,z_{0})=0$ by \cite[Proposition 4.19]{bdff}. 
But there is no boundary divisor $B$ such that the pair $(Z,B)$ is lc. 
For argument using notions of volumes, see \cite{yzhang}. 
\end{itemize}
\end{rem}

From now on, we prove the main result of this paper.

\begin{thm}\label{thmbirat}
Let $\langle X,\Delta \rangle$ be a pair such that $\Delta$ is a boundary $\mathbb{R}$-divisor. 
Then, there is a projective birational morphism $h\colon W\to X$ from a normal variety $W$ such that
\begin{itemize}
\item
any $h$-exceptional prime divisor $E_{h}$ satisfies $\alpha(E_{h},X,\Delta)<-1$, 
\item
the reduced $h$-exceptional divisor $E_{\rm red}$ is $\mathbb{Q}$-Cartier, and
\item
if we put $\Delta_{W}=h_{*}^{-1}\Delta+E_{\rm red}$, then $K_{W}+\Delta_{W}$ is $\mathbb{R}$-Cartier and the pair $(W,\Delta_{W})$ is lc. 
\end{itemize}
\end{thm}

\begin{proof}
We prove it in several steps. 
\begin{step2}\label{step2.1}
In this step, we construct a special log resolution of $\langle X,{\rm Supp}\Delta\rangle$ used in this proof. 

Let $f\colon Y\to X$ be a log resolution of $\langle X,{\rm Supp}\Delta\rangle$, and let $\Gamma$ be the sum of $f_{*}^{-1}\Delta$ and the reduced $f$-exceptional divisor.  
Let $G$ be the reduced divisor on $Y$ which is the sum of all $f$-exceptional prime divisors whose discrepancy $\alpha(\,\cdot\,,X,\Delta)$ is less than $-1$. 
By construction, we have $\Gamma-G\geq0$ and $\alpha(D,X,\Delta)\geq -1$ for any component $D$ of $\Gamma-G$ (see Lemma \ref{lembasic} (ii)). 
Suppose that there is an lc center $S_{0}$ of $(Y,\Gamma-G)$ such that for any prime divisor $P_{0}$ over $Y$ with $c_{Y}(P_{0})=S_{0}$ and $a(P_{0},Y,\Gamma-G)=-1$, we have $\alpha(P_{0},X,\Delta)<-1$. 
We take the blow-up $f_{1}\colon Y_{1}\to Y$ along $S_{0}$ and we set $\Gamma_{1}=f_{1*}^{-1}\Gamma+E_{1}$ and $G_{1}=f_{1*}^{-1}G+E_{1}$, where $E_{1}$ is the unique $f_{1}$-exceptional divisor. 
Note that $\alpha(E_{1},X,\Delta)<-1$ since we have $c_{Y}(E_{1})=S_{0}$ and $a(E_{1},Y,\Gamma-G)=-1$. 
We also see that $G_{1}$ is the sum of all $(f\circ f_{1})$-exceptional prime divisors on $Y_{1}$ whose discrepancy $\alpha(\,\cdot\,,X,\Delta)$ is less than $-1$. 
Suppose that there is an lc center $S_{1}$ of $(Y_{1},\Gamma_{1}-G_{1})$ such that for any prime divisor $P_{1}$ over $Y_{1}$ with $c_{Y_{1}}(P_{1})=S_{1}$ and $a(P_{1},Y_{1},\Gamma_{1}-G_{1})=-1$, we have $\alpha(P_{1},X,\Delta)<-1$. 
We take the blow-up $f_{2}\colon Y_{2}\to Y_{1}$ along $S_{1}$ and we set $\Gamma_{2}=f_{2*}^{-1}\Gamma_{1}+E_{2}$ and $G_{2}=f_{2*}^{-1}G_{1}+E_{2}$, where $E_{2}$ is the unique $f_{2}$-exceptional divisor. 
Then $\alpha(E_{2},X,\Delta)<-1$, and $G_{2}$ is the sum of all exceptional prime divisors over $X$ whose discrepancy $\alpha(\,\cdot\,,X,\Delta)$ is less than $-1$. 
By the standard argument, this process eventually stops. 

In this way, we obtain a log resolution $f\colon Y\to X$ of $\langle X,{\rm Supp}\Delta\rangle$, an effective $\mathbb{R}$-divisor $\Gamma$ and an effective $f$-exceptional divisor $G$ on $Y$ such that 
\begin{itemize}
\item
$\Gamma$ is the sum of $f_{*}^{-1}\Delta$ and the reduced $f$-exceptional divisor, 
\item
$G=0$ or it is a reduced divisor,
\item
for any $f$-exceptional prime divisor $E_{f}$ on $Y$, $E_{f}$ is a component of $G$ if and only if $\alpha(E_{f},X,\Delta)<-1$, and  
\item
for any lc center $S$ of $(Y,\Gamma-G)$, there is a prime divisor $Q$ over $X$ such that $c_{Y}(Q)=S$, $a(Q,Y,\Gamma-G)=-1$ and $\alpha(Q,X,\Delta)\geq-1$. 
\end{itemize}
\end{step2}

\begin{step2}\label{step2.2}
From this step to Step \ref{step2.35}, we prove that for any $0<t\leq 1$, there is the log canonical model $(W_{t}, \Gamma_{W_{t}}-tG_{W_{t}})$ of $(Y,\Gamma -tG)$ over $X$ such that any exceptional prime divisor $P$ of the morphism $W_{t}\to X$ satisfies $\alpha(P,X,\Delta)<-1$. 
We fix $0<t\leq 1$. 
Note that the conditions of $\Gamma$ and $G$ stated in Step \ref{step2.1} hold even if we restrict $f\colon Y\to X$ over an affine open subset of $X$. 
Since the log canonical model can be constructed locally, from this step to Step \ref{step2.35}, we assume that $X$ is affine. 

We run the $(K_{Y}+\Gamma-tG)$-MMP over $X$ with scaling of an ample divisor. 
After finitely many steps, we obtain a model $f'\colon (Y',\Gamma'-tG')\to X$ such that $K_{Y'}+\Gamma'-tG'$ is the limit of movable divisors over $X$, where $\Gamma'$ and $G'$ are the birational transforms of $\Gamma$ and $G$ on $Y'$, respectively. 
Then, for any $f'$-exceptional prime divisor $E'$ on $Y'$, we have $\alpha(E',X,\Delta)\leq-1$. 
Indeed, if $\alpha(E',X,\Delta)>-1$ for an $f'$-exceptional prime divisor $E'$, by Theorem \ref{proppairdiscrepancy}, there is an $\mathbb{R}$-divisor $B\geq0$ on $X$ such that $K_{X}+\Delta+B$ is $\mathbb{R}$-Cartier and $a(E',X,\Delta+B)>-1$. 
Then, 
\begin{equation*}
\begin{split}
K_{Y'}+\Gamma'-tG'=&f'^{*}(K_{X}+\Delta+B)+M'-f_{*}'^{-1}B-tG',
\end{split}
\end{equation*}
where $M'$ is an $f'$-exceptional divisor on $Y'$. 
Since $a(E',X,\Delta+B)>-1$ and $\Gamma'$ contains the reduced $f'$-exceptional divisor, the effective part of $M'$ contains $E'$ in its support. 
By construction of $G$ (see the third condition of Step \ref{step2.1} in this proof) and since $\alpha(E',X,\Delta)>-1$, we see that $E'$ is not a component of $G'$. 
Therefore, the divisor $M'-f_{*}'^{-1}B-tG'$ has non-zero effective $f'$-exceptional part.  
But it contradicts \cite[Lemma 3.3]{birkar-flip} because $K_{Y'}+\Gamma'-tG'$ is the limit of movable divisors over $X$. 
So we have $\alpha(E',X,\Delta)\leq-1$ for any $f'$-exceptional prime divisor $E'$. 

By the above argument, for any $\mathbb{R}$-divisor $C\geq 0$ on $X$ such that $K_{X}+\Delta+C$ is $\mathbb{R}$-Cartier, we can write
$$K_{Y'}+\Gamma'-tG'=f'^{*}(K_{X}+\Delta+C)-N$$
with an $N\geq0$. 
Then $a(P',Y',\Gamma'-tG')\geq a(P',X,\Delta+C)$ for any prime divisor $P'$ over $X$, where both hand sides are the usual discrepancies. 
By Theorem \ref{proppairdiscrepancy}, we have $a(P',Y',\Gamma'-tG')\geq \alpha(P',X,\Delta)$ for any prime divisor $P'$ over $X$. 
\end{step2}

\begin{step2}\label{step2.3}
We check with Theorem \ref{thmrelmmp} that $(Y',\Gamma'-tG')$ has a good minimal model over $X$. 
Note that in this step, we assume that $X$ is affine. 

It is clear that $-(K_{Y'}+\Gamma'-tG')$ is pseudo-effective over $X$. 
Pick any lc center $S'$ of $(Y',\Gamma'-tG')$. 
Then $S'$ is normal since $(Y',\Gamma'-tG')$ is $\mathbb{Q}$-factorial dlt. 
We prove that the divisor $-(K_{Y'}+\Gamma'-tG')|_{S'}$ is pseudo-effective over $X$. 
By construction, there is an lc center $S$ of $(Y,\Gamma-tG)$ such that the indeterminacy locus of the birational map $Y\dashrightarrow Y'$ does not contain $S$ and $Y\dashrightarrow Y'$ induces a birational map $S\dashrightarrow S'$. 
Since $(Y,\Gamma)$ is lc, $S$ is also an lc center of $(Y,\Gamma-G)$. 
By the fourth condition of Step \ref{step2.1} in this proof, we can find a prime divisor $Q$ over $X$ such that $c_{Y}(Q)=S$, $a(Q,Y,\Gamma-G)=-1$ and $\alpha(Q,X,\Delta)\geq-1$. 
Since $(Y,\Gamma)$ is lc, we have $a(Q,Y,\Gamma-tG)=-1$.
Since the indeterminacy locus of the map $Y\dashrightarrow Y'$ does not contain $S$, we see that $c_{Y'}(Q)=S'$ and $a(Q,Y',\Gamma'-tG')=-1$. 

Let $\overline{f}\colon \overline{Y}\to Y'$ be a log resolution of $(Y',\Gamma'-tG')$ such that $Q$ is a prime divisor on $\overline{Y}$. 
We define an $\mathbb{R}$-divisor $\Psi$ on $\overline{Y}$ by $K_{\overline{Y}}+\Psi=\overline{f}^{*}(K_{Y'}+\Gamma'-tG')$. 
We set  $\overline{f}_{Q}=\overline{f}|_{Q}\colon Q \to S'$. 
Then, $\overline{f}_{Q}$ is surjective and we have
$$-(K_{\overline{Y}}+\Psi)|_{Q}\sim_{\mathbb{R}}\overline{f}_{Q}^{*}\bigl(-(K_{Y'}+\Gamma'-tG')|_{S'}\bigr).$$ 
Therefore, to prove the pseudo-effectivity of $-(K_{Y'}+\Gamma'-tG')|_{S'}$ over $X$, it is sufficient to prove that $-(K_{\overline{Y}}+\Psi)|_{Q}$ is pseudo-effective over $X$. 

We recall that $a(Q,Y',\Gamma'-tG')\geq \alpha(Q,X,\Delta)$ (see the last sentence of Step \ref{step2.2} in this proof). 
Thus, we have  
$$-1=a(Q,Y',\Gamma'-tG')\geq \alpha(Q,X,\Delta)\geq-1,$$
and therefore, we see that $\alpha(Q,X,\Delta)=-1$. 
By Theorem \ref{proppairdiscrepancy}, for any $k\in \mathbb{Z}_{>0}$, we can find an $\mathbb{R}$-divisor $C_{k}\geq0$ on $X$ such that $K_{X}+\Delta+C_{k}$ is $\mathbb{R}$-Cartier and $a(Q,X,\Delta+C_{k})\geq -1-\frac{1}{k}$. 
We set $\beta_{k}=1+a(Q,X,\Delta+C_{k})$. 
Then $-\frac{1}{k}\leq \beta_{k}\leq0$ because we have $a(Q,X,\Delta+C_{k})\leq \alpha(Q,X,\Delta)=-1$ by Definition \ref{defnalmostdiscrepancy}. 

We recall that for any $\mathbb{R}$-divisor $C\geq 0$ on $X$ such that $K_{X}+\Delta+C$ is $\mathbb{R}$-Cartier, we can write $K_{Y'}+\Gamma'-tG'=f'^{*}(K_{X}+\Delta+C)-N$
with an $N\geq0$. 
This fact is stated in the last paragraph of Step \ref{step2.2} in this proof. 
Therefore, with an effective $\mathbb{R}$-divisor $N_{k}$ on $Y'$, we can write $K_{Y'}+\Gamma'-tG'=f'^{*}(K_{X}+\Delta+C_{k})-N_{k}$. 
By a simple calculation of discrepancies, we have 
$${\rm coeff}_{Q}(-\overline{f}^{*}N_{k})=a(Q,X,\Delta+C_{k})-a(Q,Y',\Gamma'-tG')=a(Q,X,\Delta+C_{k})+1=\beta_{k}.$$
Therefore, if we put $\overline{N}_{k}=\overline{f}^{*}N_{k}+\beta_{k}Q$, we have $\overline{N}_{k}\geq0$ and
$-\overline{f}^{*}N_{k}=\beta_{k}Q-\overline{N}_{k}.$ 
We also see that ${\rm Supp}\overline{N}_{k}\nsupseteq Q$ for any $k$ because we have ${\rm coeff}_{Q}(-\overline{f}^{*}N_{k})=\beta_{k}$. 
Furthermore, since we have $K_{\overline{Y}}+\Psi=\overline{f}^{*}(K_{Y'}+\Gamma'-tG')$ by construction, we can write 
\begin{equation*}
\begin{split}
K_{\overline{Y}}+\Psi=\overline{f}^{*}(K_{Y'}+\Gamma'-tG')&=\overline{f}^{*}f'^{*}(K_{X}+\Delta+C_{k})-\overline{f}^{*}N_{k}
\sim_{\mathbb{R},X}\beta_{k}Q-\overline{N}_{k}. 
\end{split}
\end{equation*}
From these facts, we have 
$$-(K_{\overline{Y}}+\Psi)|_{Q}+\beta_{k}Q|_{Q}\sim_{\mathbb{R},X}\overline{N}_{k}|_{Q}\geq0.$$ 
Since ${\rm lim}_{k\to \infty}\beta_{k}= 0$, we see that $-(K_{\overline{Y}}+\Psi)|_{Q}$ is pseudo-effective over $X$. 
By the argument in the third paragraph of this step, $-(K_{Y'}+\Gamma'-tG')|_{S'}$ is pseudo-effective over $X$. 
Since $S'$ is any lc center of $(Y',\Gamma'-tG')$, the morphism $(Y',\Gamma'-tG')\to X$ satisfies the hypothesis of Theorem \ref{thmrelmmp}. 
We note again that in this step, $X$ is assumed to be affine. 
In this way, we see that $(Y',\Gamma'-tG')$ has a good minimal model over $X$. 
\end{step2}

\begin{step2}\label{step2.35}
We successively assume that $X$ is affine. 
We run the $(K_{Y'}+\Gamma' - tG')$-MMP over $X$, and we get a good minimal model $(Y',\Gamma'-tG') \dashrightarrow (Y'',\Gamma''-tG'')$ over $X$. 
Let $Y'' \to W_{t}$ be the contraction over $X$ induced by $K_{Y''}+\Gamma''-tG''$. 
Because the birational map $Y\dashrightarrow Y''$ is a sequence of steps of the $(K_{Y}+\Gamma - tG)$-MMP over $X$, the pair $(W_{t}, \Gamma_{W_{t}}-tG_{W_{t}})$ is the log canonical model of $(Y,\Gamma-tG)$ over $X$, where $\Gamma_{W_{t}}$ and $G_{W_{t}}$ are the birational transforms of $\Gamma$ and $G$ on $W_{t}$, respectively. 
We denote the morphism $Y''\to X$ by $f''$. 
Now we have the following diagram.
$$
\xymatrix
{
(Y,\Gamma-tG)\ar@{-->}[r]\ar[dr]_(0.4){f}&(Y',\Gamma'-tG')\ar@{-->}[r]\ar[d]_{f'}&(Y'',\Gamma-tG'')\ar[r]\ar[dl]_(0.65){f''}&(W_{t}, \Gamma_{W_{t}}-tG_{W_{t}})\ar[dll]\\
&X
}
$$

We prove that any exceptional prime divisor $P$ of the morphism $W_{t}\to X$ satisfies $\alpha(P,X,\Delta)<-1$. 
To prove this, we prove that the morphism $Y''\to W_{t}$ contracts all $f''$-exceptional prime divisors $E''$ satisfying $\alpha(E'',X,\Delta)\geq-1$. 
By construction, $\Gamma''$ is the sum of $f_{*}''^{-1}\Delta$ and the reduced $f''$-exceptional divisor.  
We recall the third condition on $\Gamma$ and $G$ stated in Step \ref{step2.1} in this proof. 
From the condition, $E''$ is not a component of $G''$, and hence $E''$ is an lc center of $(Y'',\Gamma''-tG'')$. 
We also recall that the restriction $-(K_{Y'}+\Gamma'-tG')|_{S'}$ is pseudo-effective over $X$ for any lc center $S'$ of $(Y',\Gamma'-tG')$, which is proved in Step \ref{step2.3}. 
Therefore, by taking a common resolution of the map $Y'\dashrightarrow Y''$ and applying \cite[Lemma 4.2.10]{fujino-sp-ter}, we see that the divisor $-(K_{Y''}+\Gamma''-tG'')|_{E''}$ on $E''$ is pseudo-effective over $X$. 
On the other hand, since $(Y'',\Gamma''-tG'')$ is a good minimal model over $X$, the divisor $K_{Y''}+\Gamma''-tG''$ is semi-ample over $X$. 
From these facts, we see that the restriction of $(K_{Y''}+\Gamma''-tG'')|_{E''}$ to any sufficiently general fiber of the morphism $E''\to X$ is numerically trivial. 
This implies that the morphism $Y''\to W_{t}$ contracts all sufficiently general fibers of $E''\to X$. 
In particular, $E''$ is contracted by $Y''\to W_{t}$. 
In this way, we see that the morphism $Y''\to W_{t}$ contracts all $f''$-exceptional prime divisors $E''$ on $Y''$ satisfying $\alpha(E'',X,\Delta)\geq-1$. 
\end{step2}

\begin{step2}\label{step2.4}
In this step, $X$ is not necessarily affine. 
Let $f\colon(Y,\Gamma)\to X$ and $G$ be as in Step \ref{step2.1}. 
By steps \ref{step2.2}, \ref{step2.3} and \ref{step2.35}, for any $0<t\leq 1$, there exists the log canonical model $(W_{t}, \Gamma_{W_{t}}-tG_{W_{t}})$ of $(Y,\Gamma-tG)$ over $X$ such that any exceptional prime divisor $P$ of the morphism $W_{t}\to X$ satisfies $\alpha(P,X,\Delta)<-1$. 
Since $G_{W_{t}}$ is the birational transform of $G$ on $W_{t}$, it is the reduced exceptional divisor of $W_{t}\to X$ (see the second condition of Step \ref{step2.1} in this proof). 

Let $\{e_{n}\}_{n\geq1}$ be a strictly decreasing sequence of positive real numbers such that $e_{n}\leq 1$ and ${\rm lim}_{n\to \infty}e_{n}=0$. 
We apply Lemma \ref{lemlcmodel} to $(Y,\Gamma-e_{n}G)\to X$ and $e_{n}G$.
For each $n$, we can find $t_{n}\in(0, e_{n})$ and a birational contraction $Y\dashrightarrow W_{t_{n}}$ such that $(W_{t_{n}}, \Gamma_{W_{t_{n}}}-t_{n}G_{W_{t_{n}}})$ is the log canonical model of $(Y,\Gamma-t_{n}G)$ over $X$ and $G_{W_{t_{n}}}$ is $\mathbb{Q}$-Cartier. 
By construction, the pair $(W_{t_{n}}, \Gamma_{W_{t_{n}}}-G_{W_{t_{n}}})$ is lc, ${\rm lim}_{n\to \infty}t_{n}=0$ and the log canonical threshold ${\rm lct}(W_{t_{n}},\Gamma_{W_{t_{n}}}-G_{W_{t_{n}}};G_{W_{t_{n}}})$ is not less than $1-t_{n}$.  
By \cite[Theorem 1.1]{hmx-acc}, we can find $n$ such that ${\rm lct}(W_{t_{n}},\Gamma_{W_{t_{n}}}-G_{W_{t_{n}}};G_{W_{t_{n}}})=1$. 
For this $n$, put $W=W_{t_{n}}$, $\Delta_{W}=\Gamma_{W_{t_{n}}}$ and $G_{W}=G_{W_{t_{n}}}$. 
We denote the morphism $W\to X$ by $h$. 

We check that $h\colon (W,\Delta_{W})\to X$ satisfies all the conditions of Theorem \ref{thmbirat}. 
The first condition  of Theorem \ref{thmbirat} follows from construction of $h\colon W\to X$ (see steps \ref{step2.2}, \ref{step2.3} and \ref{step2.35}, or the third sentence of this step). 
Recall that $G_{W}$ is the reduced $h$-exceptional divisor (see the last sentence in the first paragraph of this step). 
We put $E_{\rm red}=G_{W}$, which is $\mathbb{Q}$-Cartier. 
Therefore, $E_{\rm red}$ satisfies the second condition of Theorem \ref{thmbirat}. 
We have $\Delta_{W}=h_{*}^{-1}\Delta+E_{\rm red}$ by construction of $\Gamma$ in Step \ref{step2.1} in this proof. 
Since $K_{W}+\Delta_{W}-E_{\rm red}$ is $\mathbb{R}$-Cartier and ${\rm lct}(W,\Delta_{W}-E_{\rm red};E_{\rm red})=1$, the third condition of Theorem \ref{thmbirat} is satisfied. 
\end{step2}
So we complete the proof.
\end{proof}

\begin{rem}
The proof of Theorem \ref{thmbirat} shows that for any pair $\langle X,\Delta \rangle$ such that $\Delta$ is a boundary $\mathbb{R}$-divisor and any $t>0$, we can construct $h\colon W\to X$ as in Theorem \ref{thmbirat} such that $K_{W}+\Delta_{W}-t' E_{{\rm red}}$ is $h$-ample for some $t'\in(0,t)$.   
Indeed, when we carry out the argument in Step \ref{step2.4}, we pick a strictly decreasing sequence $\{e_{n}\}_{n\geq1}$  so that $e_{1}<t$. 
With notations as in Step \ref{step2.4}, by construction of $h\colon W\to X$ the pair $(W,\Delta_{W}-t_{n}E_{{\rm red}})$ is the log canonical model of $(Y,\Gamma-t_{n}G)$ over $X$ for some $t_{n}\in (0,e_{n})$. 
Since $e_{n}<e_{1}<t$, putting $t'=t_{n}$ the divisor $K_{W}+\Delta_{W}-t' E_{{\rm red}}$ is $h$-ample. 
\end{rem}

\begin{cor}\label{corsurface}
Let $\langle X,\Delta \rangle$ be a pair. 
If $X$ is a surface, then $\langle X,\Delta \rangle$ is pseudo-lc if and only if $K_{X}+\Delta$ is $\mathbb{R}$-Cartier and $(X,\Delta)$ is lc. 
 \end{cor}

Note that $(W,\Delta_{W})$ in Theorem \ref{thmbirat} is not an lc modification of $\langle X,\Delta \rangle$. 
If there is an lc modification $(X',\Delta')$ of $\langle X,\Delta \rangle$, then the pair $(X',\Delta')$ satisfies the first and third conditions of Theorem \ref{thmbirat}. 

By the arguments in steps \ref{step2.2}, \ref{step2.3} and \ref{step2.35} in the proof of Theorem \ref{thmbirat}, we obtain the following theorem:

\begin{thm}\label{thmsmalllc}
Let $\langle X,\Delta \rangle$ be a pair such that $\Delta$ is a boundary $\mathbb{R}$-divisor. 
Let $f\colon Y\to X$ be a log resolution of $\langle X,\Delta \rangle$, and let $\Gamma$ be the sum of $f_{*}^{-1}\Delta$ and the reduced $f$-exceptional divisor. 
Suppose that 
\begin{itemize}
\item
for any $f$-exceptional prime divisor $E$, we have $\alpha(E,X,\Delta)\geq-1$, and 
\item
for any lc center $S$ of $(Y,\Gamma)$, there is a prime divisor $Q$ over $X$ such that $c_{Y}(Q)=S$, $a(Q,Y,\Gamma)=-1$ and $\alpha(Q,X,\Delta)\geq-1$. 
\end{itemize}
Then, $(Y,\Gamma)$ has the log canonical model $(W,\Delta_{W})$ over $X$ such that the natural morphism $W\to X$ is small.  
In particular, if $\langle X,\Delta \rangle$ is pseudo-lc, then there is an lc modification $h\colon(W,\Delta_{W})\to X$ such that $h$ is small. 
\end{thm}

\begin{proof}
Let $\langle X,\Delta \rangle$ and $f\colon Y\to X$ be as in the theorem. 
Then we are in the same situation as the case where $G=0$ in the final paragraph of Step \ref{step2.1} in the proof of Theorem \ref{thmbirat}. 
So the arguments in steps \ref{step2.2}, \ref{step2.3} and \ref{step2.35} in the proof of Theorem \ref{thmbirat} work with no changes. 
\end{proof}

We would like to remark about lc modifications of $\langle X,\Delta \rangle$. 
If $\Delta$ is a $\mathbb{Q}$-divisor and $K_{X}+\Delta$ is $\mathbb{Q}$-Cartier, an lc modification of $\langle X,\Delta \rangle$ exists (\cite[Theorem 1.1]{ox-lcmodel}). 
But, as we see in Example \ref{exammodel} below, the existence of lc modifications for non-$\mathbb{Q}$-Cartier pairs is in general a very difficult problem. 

\begin{exam}[{\cite[Proof of Lemma 3.2]{fg-lcring}}]\label{exammodel}
Let $X$ be a smooth projective variety such that $K_{X}$ is pseudo-effective. 
Let $A$, $f\colon Y\to X$, $S$ and $\pi\colon Y\to Z$ be as in Example \ref{examnotlc}. 
By construction, we have $K_{Y}+2S+f^{*}A\sim_{\mathbb{Q},Z}K_{Y}+S$. 
Since $S\simeq X$, we have $\kappa(S)=\kappa(X)$, where both hand sides are Kodaira dimensions. 

Suppose that the pair $\langle Z,0 \rangle$ has an lc modification $(Z',\Delta_{Z'})\to Z$. 
Then, $\Delta_{Z'}$ is the reduced exceptional divisor over $Z$, $K_{Z'}+\Delta_{Z'}$ is $\mathbb{Q}$-Cartier and ample over $Z$, and $(Z',\Delta_{Z'})$ is lc. 
We show that $(Y,S)$ has a good minimal model over $Z$. 
Indeed, since $(Z',\Delta_{Z'})$ is lc, we have $a(S,Z',\Delta_{Z'})\geq-1= a(S,Y,S)$. 
By taking a common resolution of the birational map $Y\dashrightarrow Z'$ and by the negativity lemma, we see that $(Z',\Delta_{Z'})$ is a weak lc model of $(Y,S)$ over $X$ with relatively ample log canonical divisor. 
Note that $S$ is the unique exceptional divisor of the map $Y\dashrightarrow Z'$ because $S$ is the unique exceptional divisor of $\pi$. 
By \cite[Remark 2.10]{has-mmp}, we see that $(Y,S)$ has a good minimal model over $Z$. 

Let $(Y,S)\dashrightarrow (Y',S')$ be a sequence of steps of the $(K_{Y}+S)$-MMP over $Z$ to a good minimal model. 
Then, $S$ is not contracted by the log MMP because $K_{S}$ is pseudo-effective. 
Furthermore, we have $(K_{Y'}+S')|_{S'}=K_{S'}$ and $\kappa(S)\geq \kappa(S')$ by construction. 
Since $K_{S'}$ is semi-ample, we have $\kappa(S')\geq0$. 
Then $\kappa(X)\geq0$. 

In this way, the existence of an lc modification of $\langle Z,0 \rangle$ implies the non-vanishing theorem for $X$. 
\end{exam}
By using Theorem \ref{thmsmalllc}, we see that two important theorems for lc pairs hold true in the setting of pseudo-lc pairs.

\begin{thm}\label{thmfinite}
Let $\langle X,\Delta \rangle$ be a pseudo-lc pair such that $\Delta$ is a $\mathbb{Q}$-divisor. 
Then, the graded sheaf of $\mathcal{O}_{X}$-algebra $\bigoplus_{m\geq0}\mathcal{O}_{X}(\llcorner m(K_{X}+\Delta)\lrcorner)$ is finitely generated. 
If $X$ is projective and the minimal model theory holds, then the log canonical ring  
$\bigoplus_{m\geq0}H^{0}(X, \mathcal{O}_{X}(\llcorner m(K_{X}+\Delta)\lrcorner))$
is a finitely generated $\mathbb{C}$-algebra. 
\end{thm}

\begin{proof}
The first assertion follows from Theorem \ref{thmsmalllc} and \cite[Lemma 6.2]{kollar-mori}. 
Let $h\colon(W, \Delta_{W})\to \langle X,\Delta \rangle$ be as in Theorem \ref{thmsmalllc}. 
Suppose that $X$ is projective and the minimal model theory holds. 
Then, the log canonical ring of $(W,\Delta_{W})$ is finitely generated. 
Since $H^{0}(W, \mathcal{O}_{W}(\llcorner m(K_{W}+\Delta_{W})\lrcorner))\simeq H^{0}(X, \mathcal{O}_{X}(\llcorner m(K_{X}+\Delta)\lrcorner))$, the second assertion holds. 
\end{proof}

\begin{thm}[Kodaira type vanishing theorem]\label{thmkodaira}
Let $\pi\colon X\to Z$ be a projective morphism of normal varieties and $\langle X,\Delta \rangle$ be a pseudo-lc pair. 
Let $D$ be a Weil divisor on $X$ such that $D-(K_{X}+\Delta)$ is $\pi$-ample. 

Then $R^{i}\pi_{*}\mathcal{O}_{X}(D)=0$ for any $i>0$. 
\end{thm}

\begin{proof}
We put $A=D-(K_{X}+\Delta)$. 
Let $h\colon(W,\Delta_{W})\to \langle X,\Delta \rangle$ be the lc modification as in Theorem \ref{thmsmalllc}, and take a log resolution $g\colon Y\to W$ of $( W,\Delta_{W})$. 
We can write $K_{Y}+\Gamma=g^{*}(K_{W}+\Delta_{W})+E$, where $\Gamma\geq0$ and $E\geq0$ have no common components. 
We set $D_{W}=h_{*}^{-1}D=K_{W}+\Delta_{W}+h^{*}A$. 
Then $K_{Y}+\Gamma+g^{*}h^{*}A=g^{*}D_{W}+E$. 
Since $D_{W}$ is a Weil divisor, the divisor $g^{*}D_{W}+E-\llcorner (g^{*}D_{W}+E)\lrcorner$ is $g$-exceptional. 
If we set $E'=g^{*}D_{W}+E-\llcorner (g^{*}D_{W}+E)\lrcorner$, the divisor $\Gamma-E'$ is sub-boundary and the negative part of $\Gamma-E'$ is $g$-exceptional. 
Therefore, there is a $g$-exceptional Weil divisor $E''\geq 0$ such that $\Gamma-E'+E''$ is a boundary $\mathbb{R}$-divisor.
By construction, we have 
$$K_{Y}+(\Gamma-E'+E'')+g^{*}h^{*}A=\llcorner (g^{*}D_{W}+E)\lrcorner+E''$$
and 
$$(h\circ g)_{*}\mathcal{O}_{Y}(\llcorner (g^{*}D_{W}+E)\lrcorner+E'')=h_{*}\mathcal{O}_{W}(D_{W})=\mathcal{O}_{X}(D),$$
where the second equality follows from that $h$ is small. 
By \cite[Theorem 5.6.2 (ii)]{fujino-book}, we have $R^{i}\pi_{*}R^{j}(h\circ g)_{*}\mathcal{O}_{Y}(\llcorner (g^{*}D_{W}+E)\lrcorner+E'')=0$ for any $i>0$ and $j\geq0$. 
By considering the case when $j=0$, we have $R^{i}\pi_{*}\mathcal{O}_{X}(D)=0$ for any $i>0$. 
\end{proof}

\section{A criterion of log canonicity}\label{sec5}

In this section, we discuss about a sufficient condition of log canonicity. 

\begin{thm}\label{thm5.1}
Let $X$ be a normal quasi-projective variety, and let $\Delta$ be a boundary $\mathbb{R}$-divisor. 
\begin{enumerate}
\item[(1)]
There is $\mathfrak{D}_{1}$ a finite set of prime divisors over $X$ such that if 
$${\rm sup}\{a(P,X,\Delta+G)|\,G\geq0, K_{X}+\Delta+G{\rm \; is\;}\mathbb{R}{\rm \mathchar`-Cartier}\}\geq-1$$ 
for all $P\in \mathfrak{D}_{1}$, then $\langle X,\Delta \rangle$ has a small lc modification. 
In particular, when $\Delta$ is a $\mathbb{Q}$-divisor, the graded sheaf of $\mathcal{O}_{X}$-algebra $\bigoplus_{m\geq0}\mathcal{O}_{X}(\llcorner m(K_{X}+\Delta)\lrcorner)$ is finitely generated. 
\item[(2)]
Suppose that $\langle X,\Delta \rangle$ has a small lc modification. 
Let $x \in X$ be a closed point. 
Then, there is $\mathfrak{D}_{2}$ a finite set of prime divisors over $X$ such that $K_{X}+\Delta$ is $\mathbb{R}$-Cartier and $( X,\Delta )$ is lc in a neighborhood of $x$ if and only if the following relation holds for any $P\in \mathfrak{D}_{2}$. 
\begin{equation*} \begin{split} &{\rm sup}\{a(P,X,\Delta+G)|\,G\geq0, K_{X}+\Delta+G{\rm \; is\;}\mathbb{R}{\rm \mathchar`-Cartier}\}\\ =&{\rm inf}\{a(P,X,\Delta-G')|\,G'\geq0, K_{X}+\Delta-G'{\rm \; is\;}\mathbb{R}{\rm \mathchar`-Cartier}\}. \end{split} \end{equation*}
\end{enumerate}
\end{thm}

\begin{proof}
First, we prove (1). 
Let $f\colon Y\to X$ be a log resolution of  $\langle X,{\rm Supp}\Delta \rangle$. 
Let $\{E_{i}\}_{i}$ be the set of all $f$-exceptional prime divisors. 
We set $\Gamma=f_{*}^{-1}\Delta+\sum_{i}E_{i}$.  
For any lc center $S$ of $(Y,\Gamma)$, fix a prime divisor $E'_{S}$ over $X$ such that the center of $E'_{S}$ on $Y$ is $S$ and $a(E'_{S},Y,\Gamma)=-1$. 
We set $\mathfrak{D}_{1}=\{E_{i}\}_{i}\cup \{E'_{S}\}_{S}$, where $S$ runs over all lc centers of $(Y,\Gamma)$. 
Then $\mathfrak{D}_{1}$ satisfies the condition of Theorem \ref{thm5.1} (1). 
Indeed, if 
$${\rm sup}\{a(P,X,\Delta+G)|\,G\geq0, K_{X}+\Delta+G{\rm \; is\;}\mathbb{R}{\rm \mathchar`-Cartier}\}\geq-1$$ 
for any $P\in \mathfrak{D}_{1}$, by Theorem \ref{proppairdiscrepancy}, the morphism $f\colon (Y,\Gamma)\to\langle X,\Delta\rangle $ satisfies the conditions of Theorem \ref{thmsmalllc}. 
So, $(Y,\Gamma)$ has the log canonical model $(W,\Delta_{W})$ over $X$ such that the induced morphism $h\colon W\to X$ is small. 
By Definition \ref{defnlcmodification}, $(W,\Delta_{W})$ is a small lc modification of  $\langle X,\Delta \rangle$. 
Thus, we complete the proof of (1). 

Next, we prove (2). 
Let $h\colon(W,\Delta_{W})\to \langle X,\Delta \rangle$ be a small lc modification. 
Note that $\Delta_{W}=h_{*}^{-1}\Delta$. 
Let $f\colon Y\to X$ be a log resolution of  $\langle X,{\rm Supp}\Delta \rangle$ such that $f^{-1}(x)$ is a simple normal crossing divisor and the induced map $g\colon Y\dashrightarrow W$ is a morphism. 
Let $\mathfrak{D}_{2}$ be the set of all $f$-exceptional prime divisors whose centers on $X$ contain $x$. 
We prove that $\mathfrak{D}_{2}$ satisfies the condition of Theorem \ref{thm5.1} (2). 
Since Theorem \ref{thm5.1} (2) is a local problem, shrinking $X$, we may assume that $\mathfrak{D}_{2}$ is the set of all $f$-exceptional prime divisors. 
Suppose that $K_{X}+\Delta$ is $\mathbb{R}$-Cartier and $(X,\Delta )$ is lc in a neighborhood of $x$. 
By shrinking $X$ again, we can assume $(X,\Delta)$ is lc. 
Then, as in the proof of Lemma \ref{lembasic}, for any $P\in \mathfrak{D}_{2}$, we obtain
\begin{equation*}
\begin{split}
&{\rm sup}\{a(P,X,\Delta+G)|\,G\geq0, K_{X}+\Delta+G{\rm \; is\;}\mathbb{R}{\rm \mathchar`-Cartier}\}\\
=&{\rm inf}\{a(P,X,\Delta-G')|\,G'\geq0, K_{X}+\Delta-G'{\rm \; is\;}\mathbb{R}{\rm \mathchar`-Cartier}\}\\
=&a(P,X,\Delta)\geq-1. 
\end{split}
\end{equation*}
Therefore, the first condition implies the second condition. 

Conversely, suppose that the equation as in Theorem \ref{thm5.1} (2) holds for all $P\in\mathfrak{D}_{2}$. 
We check
$a(P,W,\Delta_{W})=\alpha(P,X,\Delta)$ for any $P\in\mathfrak{D}_{2}$, where $\alpha(\,\cdot\,,X,\Delta)$ is as in Definition \ref{defnalmostdiscrepancy}. 
By Theorem \ref{proppairdiscrepancy}, we only have to show
\begin{equation}\tag{$*$}
\begin{split}
&{\rm sup}\{a(P,X,\Delta+G)|\,G\geq0, K_{X}+\Delta+G{\rm \; is\;}\mathbb{R}{\rm \mathchar`-Cartier}\}\\\leq& a(P,W,\Delta_{W})\\
\leq&{\rm inf}\{a(P,X,\Delta-G')|\,G'\geq0, K_{X}+\Delta-G'{\rm \; is\;}\mathbb{R}{\rm \mathchar`-Cartier}\}
\end{split}
\end{equation}
for any $P\in\mathfrak{D}_{2}$. 
We only show the first relation because the second relation can be obtained similarly. 
Since $h$ is small, for any $G\geq0$ on $X$ such that $K_{X}+\Delta+G$ is $\mathbb{R}$-Cartier, we have $h^{*}(K_{X}+\Delta+G)-(K_{W}+\Delta_{W})\geq0$. 
By \cite[Lemma 2.27]{kollar-mori}, we have $a(P,X,\Delta+G)\leq a(P,W,\Delta_{W})$ for any $G$ and $P\in\mathfrak{D}_{2}$. 
Therefore, the first relation of ($*$) holds. 
Thus, we see that the relation ($*$) holds. 
In this way, we have $a(P,W,\Delta_{W})=\alpha(P,X,\Delta)$ for any $P\in\mathfrak{D}_{2}$.  

From now on, we prove that $K_{X}+\Delta$ is $\mathbb{R}$-Cartier and $(X,\Delta )$ is lc near $x$. 
Since $(W,\Delta_{W})$ is lc by construction, it is sufficient to prove that the morphism $h\colon W\to X$ is an isomorphism over $x$. 
To prove this, we only have to show that $h^{-1}(x)$ is a point. 
Suppose by contradiction that $h^{-1}(x)$ is not a point. 
Then $h^{-1}(x)$ contains a curve. 
We may write
$K_{Y}+f_{*}^{-1}\Delta-\sum_{P\in\mathfrak{D}_{2}}a(P,W,\Delta_{W})P=g^{*}(K_{W}+\Delta_{W}).$
Since $a(P,W,\Delta_{W})=\alpha(P,X,\Delta)$, for any integer $n>0$, we can find $G_{n}\geq0$ on $X$ such that $K_{X}+\Delta+G_{n}$ is $\mathbb{R}$-Cartier and 
$a(P,W,\Delta_{W})-a(P,X,\Delta+G_{n})\in[0, \frac{1}{n}]$
for any $P\in\mathfrak{D}_{2}$ (Theorem \ref{proppairdiscrepancy}).
Therefore, if we put $\beta_{n,P}=a(P,W,\Delta_{W})-a(P,X,\Delta+G_{n})$ for any $P\in\mathfrak{D}_{2}$, we have $0\leq \beta_{n,P} \leq \tfrac{1}{n}$. 
Moreover, we can write
\begin{equation*}
\begin{split}
g^{*}(K_{W}+\Delta_{W})&=K_{Y}+f_{*}^{-1}\Delta-\sum_{P\in\mathfrak{D}_{2}}\bigl(a(P,X,\Delta+G_{n})+\beta_{n,P}\bigr)P\\
&=f^{*}(K_{X}+\Delta +G_{n})-f_{*}^{-1}G_{n}-\sum_{P\in\mathfrak{D}_{2}}\beta_{n,P}P\\
&\sim_{\mathbb{R},X}-f_{*}^{-1}G_{n}-\sum_{P\in\mathfrak{D}_{2}}\beta_{n,P}P.
\end{split}
\end{equation*}
Now we recall that $g^{-1}(h^{-1}(x))=f^{-1}(x)$ is a simple normal crossing divisor and $h^{-1}(x)$ contains a curve. 
Pick any sufficiently general curve $\xi\subset f^{-1}(x)$ such that $g(\xi)$ is a curve on $W$. 
Then, for any $n$, we obtain 
$$(g^{*}(K_{W}+\Delta_{W}))\,\cdot \,\xi=-(f_{*}^{-1}G_{n}\,\cdot\,\xi)-\sum_{P\in\mathfrak{D}_{2}}\beta_{n,P}(P\,\cdot\, \xi)\leq-\sum_{P\in\mathfrak{D}_{2}}\beta_{n,P}(P\,\cdot\, \xi).$$
Since $0\leq \beta_{n,P} \leq \tfrac{1}{n}$, considering the limit $n\to \infty$, we have $(g^{*}(K_{W}+\Delta_{W}))\,\cdot\, \xi\leq0$. 
Now recall that $(W,\Delta_{W})$ is an lc modification of $\langle X,\Delta \rangle$. 
So $K_{W}+\Delta_{W}$ is ample over $X$. 
Since $g(\xi)$ is a curve on $W$ and $f(\xi)=x$, we have $(g^{*}(K_{W}+\Delta_{W}))\,\cdot \,\xi>0$. 
In this way, we get a contradiction. 

In this way, we see that $h^{-1}(x)$ is a point. 
So $h$ is an isomorphism over $x$, and $K_{X}+\Delta$ is $\mathbb{R}$-Cartier and $(X,\Delta )$ is lc near $x$. 
Thus we complete the proof. 
\end{proof}

\begin{cor}\label{corlccriterion}
Let $\langle X,\Delta \rangle$ be a pair such that $X$ is quasi-projective. 
Then $K_{X}+\Delta$ is $\mathbb{R}$-Cartier and $( X,\Delta )$ is lc if and only if $\langle X,\Delta \rangle$ is pseudo-lc and the following equation holds for any prime divisor $P$ over $X$. 
\begin{equation*}
\begin{split}
&{\rm sup}\{a(P,X,\Delta+G)|\,G\geq0, K_{X}+\Delta+G{\rm \; is\;}\mathbb{R}{\rm \mathchar`-Cartier}\}\\
=&{\rm inf}\{a(P,X,\Delta-G')|\,G'\geq0, K_{X}+\Delta-G'{\rm \; is\;}\mathbb{R}{\rm \mathchar`-Cartier}\}.
\end{split}
\end{equation*} 
\end{cor}

\begin{proof}
It follows from Theorem \ref{thm5.1}. 
\end{proof}

In the rest of this paper, we give an other proof of the Corollary \ref{corlccriterion}. 
We prove it using the notion of numerically Cartier divisors (see \cite[Definition 5.2]{bdffu}) and the minimal model theory. 

First, we recall the notion of numerically Cartier divisors. 

\begin{defn}[{\cite[Definition 5.2]{bdffu}}, see also {\cite[Definition 2.26]{bdff}}]
Let $X$ be a normal variety, and let $D$ be an $\mathbb{R}$-divisor on it. 
Then $D$ is {\em numerically} {\em Cartier} if there is a resolution $f\colon Y\to X$ and an $\mathbb{R}$-divisor $D_{Y}$ on $Y$ such that $D_{Y}$ is numerically trivial over $X$ and $f_{*}D_{Y}=D$. 
\end{defn}

The following lemma connects the notion of numerically Cartier divisors and the usual discrepancy. 

\begin{lem}\label{lemnumercartier}
Let $\langle X,\Delta \rangle$ be a pair such that $X$ is quasi-projective. 
Then, $K_{X}+\Delta$ is numerically Cartier if and only if for any prime divisor $P$ over $X$, the following equality holds: 
\begin{equation*}
\begin{split}
&{\rm sup}\{a(P,X,\Delta+G)|\,G\geq0, K_{X}+\Delta+G{\rm \; is\;}\mathbb{R}{\rm \mathchar`-Cartier}\}\\
=&{\rm inf}\{a(P,X,\Delta-G')|\,G'\geq0, K_{X}+\Delta-G'{\rm \; is\;}\mathbb{R}{\rm \mathchar`-Cartier}\}.
\end{split}
\end{equation*} 
\end{lem}

\begin{proof}
We use notations in \cite{bdff}. 
Let ${\rm Env}_{X}(\,\cdot\,)$ be the nef envelope.  
We set 
\begin{equation*}
\begin{split}
\alpha'(P,X,\Delta)={\rm inf}\{a(P,X,\Delta-G')|\,G'\geq0, K_{X}+\Delta-G'{\rm \; is\;}\mathbb{R}{\rm \mathchar`-Cartier}\}, 
\end{split}
\end{equation*} 
and for any log resolution $f\colon Y \to X$ of  $\langle X,{\rm Supp}\Delta \rangle$, we put $D_{Y}=\sum_{P}\alpha(P,X,\Delta)$ and $D'_{Y}=\sum_{P}\alpha'(P,X,\Delta)$, where $P$ runs over all prime divisors on $Y$. 
Then, we have $D_{Y}=K_{Y}+({\rm Env}_{X}(-(K_{X}+\Delta)))_{Y}$ by Theorem \ref{proppairdiscrepancy} and Theorem \ref{thmnefenvelope}. 
By the same arguments, we also obtain $D'_{Y}=K_{Y}-({\rm Env}_{X}(K_{X}+\Delta))_{Y}$.
Therefore, the equality 
\begin{equation*}
\begin{split}
&{\rm sup}\{a(P,X,\Delta+G)|\,G\geq0, K_{X}+\Delta+G{\rm \; is\;}\mathbb{R}{\rm \mathchar`-Cartier}\}\\
=&{\rm inf}\{a(P,X,\Delta-G')|\,G'\geq0, K_{X}+\Delta-G'{\rm \; is\;}\mathbb{R}{\rm \mathchar`-Cartier}\}
\end{split}
\end{equation*} is equivalent to ${\rm Env}_{X}(-(K_{X}+\Delta))=-{\rm Env}_{X}(K_{X}+\Delta)$ as $b$-$\mathbb{R}$-divisors. 
But this is equivalent to that $K_{X}+\Delta$ is numerically Cartier. 
For details, see the proof of \cite[Proposition 5.9]{bdffu}. 
Note that the proof of \cite[Proposition 5.9]{bdffu} is carried out with $\mathbb{Q}$-divisors, but the argument works for $\mathbb{R}$-divisors without any change. 
\end{proof}

By Lemma \ref{lemnumercartier}, Corollary \ref{corlccriterion} is equivalent to the following statement, which is an lc analog of \cite[Corollary 5.17]{bdffu}. 

\begin{thm}
Let $\langle X,\Delta \rangle$ be a pair such that $X$ is quasi-projective. 
Suppose that $K_{X}+\Delta$ is numerically Cartier. 
Suppose in addition that for any log resolution $f \colon Y\to X$ of $\langle X,\Delta \rangle$, the coefficient of any $P$ in $K_{Y}+({\rm Env}_{X}(-(K_{X}+\Delta)))_{Y}$ is not less
than $-1$. 

Then, $K_{X}+\Delta$ is $\mathbb{R}$-Cartier and $(X,\Delta)$ is lc. 

\end{thm}

\begin{proof}
Fix a log resolution $f \colon Y\to X$ of $\langle X,\Delta \rangle$. 
Then, we may write 
$$K_{Y}+f_{*}^{-1}\Delta+E_{+}-E_{-}=-({\rm Env}_{X}(-(K_{X}+\Delta)))_{Y},$$
where $E_{+}$ and $E_{-}$ are effective $f$-exceptional $\mathbb{R}$-divisors which have no common components, and $E_{+}$ is a boundary divisor. 
Since $K_{X}+\Delta$ is numerically Cartier, we see that $K_{Y}+f_{*}^{-1}\Delta+E_{+}-E_{-}$ is numerically trivial over $X$. 

We set $\Gamma=f_{*}^{-1}\Delta+E_{+}$. 
Then $(Y,\Gamma)$ is lc. 
We run the $(K_{Y}+\Gamma)$-MMP over $X$ with scaling of an ample divisor. 
After finitely many steps, we reach a model $(Y,\Gamma)\dashrightarrow (Y',\Gamma')$ over $X$ such that $K_{Y'}+\Gamma'$ is the limit of movable divisors over $X$. 
Let $E'_{-}$ be the birational transform of $E_{-}$ on $Y'$. 
By the contraction theorem \cite[Theorem 4.5.2 (4)]{fujino-book}, we see that $K_{Y'}+\Gamma'-E'_{-}$ is numerically trivial over $X$. 
Then, for any sufficiently general curve $\xi$ whose image on $X$ is a point, we have 
$(E'_{-}\,\cdot\, \xi)=(K_{Y'}+\Gamma')\,\cdot\, \xi\geq0$. 
Since $E'_{-}$ is effective and exceptional over $X$, by \cite[Lemma 3.3]{birkar-flip}, we have $E'_{-}=0$. 
Therefore, we see that $K_{Y'}+\Gamma'$ is numerically trivial over $X$. 
By Theorem \ref{thmrelmmp}, $K_{Y'}+\Gamma'$ is semi-ample over $X$. 
So there is the log canonical model $(W,\Delta_{W})$ of $(Y',\Gamma')$ over $X$ such that $W$ is isomorphic to $X$, where $\Delta_{W}$ is the birational transform of $\Gamma'$ on $W$. 
By construction, $\Delta_{W}$ is the birational transform of $\Delta$ on $W$. 
So we see that $K_{X}+\Delta$ is $\mathbb{R}$-Cartier and $(X,\Delta)$ is lc. 
\end{proof}



\begin{thebibliography}{BdFFU}









\bibitem[B1]{birkar-flip}
C.~Birkar, 
Existence of log canonical flips and a special LMMP, 
Publ. Math. Inst. Hautes \'Etudes Sci. {\textbf{115}} (2012), no. 1, 325--368.

\bibitem[B2]{birkar-bab}
C.~Birkar, Singularities of linear systems and boundedness of Fano varieties, preprint (2016), arXiv:1609.05543v1.


\bibitem[BCHM]{bchm}C.~Birkar, P.~Cascini, C.~D.~Hacon, J.~M\textsuperscript{c}Kernan, Existence of minimal models for varieties of log general type, J. Amer. Math. Soc. {\textbf{23}} (2010), no. 2, 405--468.







\bibitem[BZ]{bz} C.~Birkar, D.~Q.~Zhang, Effectivity of Iitaka fibrations and pluricanonical systems of polarized pairs, Publ. Math. Inst. Hautes \'Etudes Sci. {\textbf{123}} (2016), no. 1, 283--331.

\bibitem[BdFF]{bdff} S.~Boucksom, T. de Fernex, C.~Favre, 
The volume of an isolated singularity, Duke Math. J. {\textbf{161}} (2012), no. 8, 1455--1520.


\bibitem[BdFFU]{bdffu} S.~Boucksom, T. de Fernex, C.~Favre, S.~Urbinati, 
Valuation spaces and multiplier ideals on singular varieties, 
Recent advances in algebraic geometry, 29-–51,
London Math. Soc. Lecture Note Ser., {\textbf{417}}, Cambridge Univ. Press, Cambridge, 2015. 






\bibitem[dFH]{dfh} T. de Fernex, C.~D.~Hacon, Singularities on normal varieties, Compos. Math. {\textbf{145}} (2009), no. 2, 393--414. 





\bibitem[F1]{fujino-sp-ter}O.~Fujino, {\it Special termination and reduction to pl flips.} In Flips for $3$-folds and $4$-folds, Oxford University Press (2007).


\bibitem[F2]{fujinonon-van}O.~Fujino, Non-vanishing theorem for log canonical pairs, J. Algebraic Geom. {\textbf{20}} (2011), no. 4, 771--783. 

\bibitem[F3]{fujino-fund}O.~Fujino, Fundamental theorems for the log minimal model program, Publ. Res. Inst. Math. Sci. {\textbf{47}} (2011), no. 3, 727--789. 







\bibitem[F4]{fujino-book}O.~Fujino, {\em Foundations of the minimal model program}, MSJ Mem. \textbf{35}, Mathematical Society in Japan, Tokyo, 2017. 





\bibitem[FG1]{fujino-gongyo}O.~Fujino, Y.~Gongyo, Log pluricanonical representations and abundance conjecture, Compos. Math. {\textbf{150}} (2014) no. 4, 593--620.




\bibitem[FG2]{fg-lcring} O.~Fujino, Y.~Gongyo, On log canonical rings, Adv. Stud. Pure Math., {\textbf{74}} (2017), Higher dimensional algebraic geometry in honour of Professor Yujiro Kawamata's sixtieth birthday, 159--169,
















\bibitem[HMX1]{hmx-acc}C.~D.~Hacon, J.~M\textsuperscript{c}Kernan, C.~Xu, ACC for log canonical thresholds, Ann. of Math. (2) {\textbf{180}} (2014), no. 2, 523--571. 

\bibitem[HMX2]{hmx} C.~D.~Hacon, J.~M\textsuperscript{c}Kernan, C.~Xu, Boundedness of moduli of varieties of general type, J. Eur. Math. Soc. {\textbf{20}} (2018), no. 4, 865--901. 

\bibitem[HX]{haconxu-lcc}C.~D.~Hacon, C.~Xu, Existence of log canonical closures, Invent.Math. {\textbf{192}} (2013), no. 1, 161--195. 




\bibitem[H1]{has-trivial}
K.~Hashizume, 
Minimal model theory for relatively trivial log canonical pairs, Ann. Inst. Fourier (Grenoble) {\textbf{68}} (2018), no. 5, 2069--2107 .


\bibitem[H2]{has-mmp}
K.~Hashizume, 
Remarks on special kinds of the relative log minimal model program, to appear in Manuscripta Math.















\bibitem[K]{kollar-logpluri} J.~Koll\'ar, Log-plurigenera in stable families, preprint (2018), arXiv:1801.05414v2.





\bibitem[KM]{kollar-mori} J.~Koll\'ar, S.~Mori, {\em{Birational geometry of algebraic varieties}}. With the collaboration of C.~H.~Clemens and A.~Corti. Translated from the 1998 Japanese original. Cambridge Tracts in Mathematics, {\textbf{134}}. Cambridge University Press, Cambridge, 1998.









\bibitem[N]{nakayama}N.~Nakayama, {\em Zariski-decomposition and abundance}, MSJ Mem., {\textbf{14}}, Mathematical Society of Japan, Tokyo, 2004. 


\bibitem[OX]{ox-lcmodel} Y.~Odaka, C.~Xu, Log-canonical models of singular pairs and its applications, Math. Res. Lett., {\textbf{19}} (2012), no. 2, 325--334. 



\bibitem[S]{shokurovcompl} V.~V.~Shokurov, Complements on surfaces, Algebraic geometry, 10. J. Math. Sci. (New York) {\textbf{102}} (2000), no. 2, 3876--3932. 

\bibitem[Z]{yzhang} Y.~Zhang, On the volume of isolated singularities, Compos. Math. {\textbf{150}} (2014), no. 8, 1413--1424. 

\end{thebibliography}
\end{document}